\documentclass[12pt,a4paper]{article}
\usepackage{amssymb}

\usepackage{amsthm}

\newtheorem{proposition}{Proposition}[section]
\newtheorem{theorem}{Theorem}[section]
\newtheorem{remark}{Remark}[section]
\newtheorem{corollary}{Corollary}[section]
\newtheorem{lemma}{Lemma}[section]

\begin{document}

\author{Nicoleta Aldea and Gheorghe Munteanu}
\title{On two - dimensional complex Finsler manifolds}
\date{}
\maketitle

\begin{abstract}
In this paper, we investigate the two-dimensional complex Finsler manifolds.
The tools of this study are the complex Berwald frames $\{l,$ $m,$ $\bar{l},$
$\bar{m}\}$, $\{\lambda ,\mu ,\bar{\lambda},\bar{\mu}\}$ and the
Chern-Finsler connection with respect to these frames.

The geometry of two-dimensional complex Finsler manifolds is controlled by
three real invariants which live on $T^{\prime }M$: two horizontal curvature
invariants $\mathbf{K}$ and $\mathbf{W}$ and one vertical curvature
invariant $\mathbf{I}.$ By means of these invariants are defined both the
horizontal and the vertical holomorphic sectional curvatures in directions $%
\lambda $, $\mu $ and $m$, respectively.

The complex Landsberg and Berwald spaces are of particular interest. Complex
Berwald spaces coincide with K\"{a}hler spaces, in the two - dimensional
case. We establish the necessary and sufficient condition so that $\mathbf{K}
$ is a constant and we obtain a characterization for the K\"{a}hler purely
Hermitian spaces by the fact $\mathbf{K}=\mathbf{W}=$ constant and $\mathbf{I%
}=0$. For the class of complex Berwald spaces we have $\mathbf{K}=\mathbf{W}%
=0$. Finally, a classification of two - dimensional complex Finsler spaces
for which the horizontal curvature satisfies a special property is obtained.
\end{abstract}

\begin{flushleft}
\strut \textbf{2000 Mathematics Subject Classification:} 53B40, 53C60.

\end{flushleft}

\section{Introduction}

\setlength\arraycolsep{3pt} \setcounter{equation}{0}A great contribution to
the geometry of two-dimensional real Finsler spaces is due to L. Berwald (%
\cite{Bw}, see also \cite{Ma}). His theory is developed based on the choice
of an orthonormal frame consisting of the normalized Liouville field and a
unit field orthogonal to it. Many remarkable results are known for
two-dimensional real Finsler spaces (\cite{Ma,As,B-S,Be,Sh}).

Part of the general themes from real Finsler geometry can be approached in
complex Finsler geometry, the complex setting having the advantage of a
powerful instrument, namely the Chern-Finsler connection (cf. \cite{A-P}).
This connection is Hermitian, of $(1,0)$ - type and with other special
properties, but as we expect, there are some uncomfortable computations on
account of extending the theory to the complexified holomorphic tangent
bundle $T^{\prime }M.$

In a previous paper \cite{Mu-Al}, we constructed the vertical Berwald frame
in which the orthogonality is, with respect to the Hermitian structure,
defined by the fundamental metric tensor of a 2 - dimensional complex
Finsler space, on the holomorphic tangent manifold $T^{\prime }M.$

The main purpose of this paper is to give a partial classification of the 2
- dimensional complex Finsler manifolds using its Chern-Finsler curvatures.
We do not give a general complete classification, but we emphasize some
important particular classes of the $2$ - dimensional complex Finsler spaces.

Subsequently, we have made an overview of the paper's content.

In \S 2, we recall some preliminary properties of the $n$- dimensional
complex Finsler spaces and complete with some others needed.

In \S 3, we prepare the tools for our aforementioned study. After we review
from \cite{Mu-Al} the construction of the Berwald frame of a complex Finsler
manifold of dimension two, we prefer to work in a fixed local chart in which
there is obtained a local complex Berwald frame, which is extended to one on
the horizontal part. We also find the expression of the complex
Chern-Finsler connection with respect to these local frames. The
independence of the obtained results from chosen chart is incessantly
studied.

The local Berwald frames are not only a local geometrical machinery, but
they also satisfy important properties which contain three main real
invariants which live on $T^{\prime }M:$ one vertical curvature invariant $%
\mathbf{I}$ and two horizontal curvature invariants $\mathbf{K}$ and $%
\mathbf{W}.$ By means of these invariants we are able to define and compute
the horizontal and vertical holomorphic sectional curvatures in directions $%
\lambda ,$ $\mu $ and $m$, respectively. A first classification of the
complex Finsler manifold of dimension two comes from the exploration of the $%
v\bar{v}-,$ $h\bar{v}-$ and $v\bar{h}-$ Riemann type tensors, (Theorem 4.2).
An immediate interest for the $2$ - dimensional complex Berwald spaces is
induced by the properties of the $h\bar{v}-$ and $v\bar{h}-$ Riemann type
tensors. We prove that the two dimensional complex Berwald spaces are
reduced to the K\"{a}hler spaces (Theorems 4.3, 4.4). We show that for the
complex Berwald spaces $\mathbf{I}_{|k}=0$, (Proposition 4.5). But this
property is not peculiar only to the complex Berwald spaces. An example of
the $2$ - dimensional complex Finsler metric with $\mathbf{I}_{|k}=0,$ which
is not Berwald, is given by the complex version of the Antonelli-Shimada
metric. The necessary and sufficient conditions for $2$ - dimensional
complex Landsberg spaces are given in Theorem 4.6. Next, we derive the
Bianchi identities which specify the relations among the covariant
derivatives of the three invariants and then use these relations to explore
the holomorphic sectional curvatures. With some additional conditions of
symmetry for the $h\bar{h}-$ Riemann type tensor, we find the necessary and
sufficient conditions that $\mathbf{K}$ should be a constant (Theorems 4.8,
4.9). Moreover, we characterize the spaces with $\mathbf{K}=0$ and $\mathbf{W%
}\leq 0,$ (Theorem 4.11). The complex Berwald spaces with this symmetry have
$\mathbf{K}=\mathbf{W}=0,$ (Theorem 4.12). It results that the K\"{a}hler
purely Hermitian spaces are characterized by $\mathbf{K}=\mathbf{W}=$
constant and $\mathbf{I}=0$, (Theorem 4.13). Finally, a special approach is
devoted to the spaces with the $h\bar{h}-$ Riemann type tensor in the form $%
R_{\bar{r}j\bar{h}k}=\mathcal{K}(g_{j\bar{r}}g_{k\bar{h}}+g_{k\bar{r}}g_{j%
\bar{h}})$. We obtain two classes of such spaces, namely the K\"{a}hler
purely Hermitian with $\mathcal{K}$ a constant and the complex spaces with $%
\mathcal{K}=0$ and $\dot{\partial}_{\bar{h}}G^i=0$, (Theorem 4.15). All
these results are in \S 4.

\section{Preliminaries}

\setcounter{equation}{0} In the beginning, we will make a survey of complex
Finsler geometry and we will set the basic notions and terminology. For
more, see \cite{A-P,Mub}.

Let $M$ be a $n-$dimensional complex manifold, $z=(z^k)_{k=\overline{1,n}}$
are the complex coordinates in a local chart.

The complexified of the real tangent bundle $T_CM$ splits into the sum of
holomorphic tangent bundle $T^{\prime }M$ and its conjugate $T^{\prime
\prime }M$. The bundle $T^{\prime }M$ is itself a complex manifold, and the
local coordinates in a local chart will be denoted by $u=(z^k,\eta ^k)_{k=%
\overline{1,n}}.$ They are changed into $(z^{\prime k},\eta ^{\prime k})_{k=%
\overline{1,n}}$ by the rules $z^{\prime k}=z^{\prime k}(z)$ and $\eta
^{\prime k}=\frac{\partial z^{\prime k}}{\partial z^l}\eta ^l.$

A \textit{complex Finsler space} is a pair $(M,F)$, where $F:T^{\prime
}M\rightarrow \mathbb{R}^{+}$ is a continuous function satisfying the
conditions:

\textit{i)} $L:=F^2$ is smooth on $\widetilde{T^{\prime }M}:=T^{\prime
}M\backslash \{0\};$

\textit{ii)} $F(z,\eta )\geq 0$, the equality holds if and only if $\eta =0;$

\textit{iii)} $F(z,\lambda \eta )=|\lambda |F(z,\eta )$ for $\forall \lambda
\in \mathbb{C}$;

\textit{iv)} the Hermitian matrix $\left( g_{i\bar{j}}(z,\eta )\right) $ is
positive defined, where $g_{i\bar{j}}:=\frac{\partial ^2L}{\partial \eta
^i\partial \bar{\eta}^j}$ is the fundamental metric tensor. Equivalently, it
means that the indicatrix is strongly pseudo-convex.

Consequently, from $iii$) we have $\frac{\partial L}{\partial \eta ^k}\eta
^k=\frac{\partial L}{\partial \bar{\eta}^k}\bar{\eta}^k=L,$ $\frac{\partial
g_{i\bar{j}}}{\partial \eta ^k}\eta ^k=\frac{\partial g_{i\bar{j}}}{\partial
\bar{\eta}^k}\bar{\eta}^k=0$ and $L=g_{i\bar{j}}\eta ^i\bar{\eta}^j.$

Roughly speaking, the geometry of a complex Finsler space consists of the
study of the geometric objects of the complex manifold $T^{\prime }M$
endowed with the Hermitian metric structure defined by $g_{i\bar{j}}.$

Therefore, the first step is to study the sections of the complexified
tangent bundle of $T^{\prime }M,$ which is decomposed in the sum $%
T_C(T^{\prime }M)=T^{\prime }(T^{\prime }M)\oplus T^{\prime \prime
}(T^{\prime }M)$. Let $VT^{\prime }M\subset T^{\prime }(T^{\prime }M)$ be
the vertical bundle, locally spanned by $\{\frac \partial {\partial \eta
^k}\}$, and $VT^{\prime \prime }M$ its conjugate.

At this point, the idea of complex nonlinear connection, briefly $(c.n.c.),$
is an instrument in 'linearization' of this geometry. A $(c.n.c.)$ is a
supplementary complex subbundle to $VT^{\prime }M$ in $T^{\prime }(T^{\prime
}M)$, i.e. $T^{\prime }(T^{\prime }M)=HT^{\prime }M\oplus VT^{\prime }M.$
The horizontal distribution $H_uT^{\prime }M$ is locally spanned by $\{\frac
\delta {\delta z^k}=\frac \partial {\partial z^k}-N_k^j\frac \partial
{\partial \eta ^j}\}$, where $N_k^j(z,\eta )$ are the coefficients of the $%
(c.n.c.)$. The pair $\{\delta _k:=\frac \delta {\delta z^k},\dot{\partial}%
_k:=\frac \partial {\partial \eta ^k}\}$ will be called the adapted frame of
the $(c.n.c.)$ which obey to the change rules $\delta _k=\frac{\partial
z^{\prime j}}{\partial z^k}\delta _j^{\prime }$ and $\dot{\partial}_k=\frac{%
\partial z^{\prime j}}{\partial z^k}\dot{\partial}_j^{\prime }.$ By
conjugation, everywhere is obtained an adapted frame $\{\delta _{\bar{k}},%
\dot{\partial}_{\bar{k}}\}$ on $T_u^{\prime \prime }(T^{\prime }M).$ The
dual adapted bases are $\{dz^k,\delta \eta ^k\}$ and $\{d\bar{z}^k,\delta
\bar{\eta}^k\}.$

Certainly, a main problem in this geometry is to determine a $(c.n.c.)$
related only to the fundamental function of the complex Finsler space $%
(M,F). $

The next step is the action of a derivative law $D$ on the sections of $%
T_C(T^{\prime }M).$ First, let us consider the \textit{Sasaki} type lift of
the metric tensor $g_{i\bar{j}},$%
\begin{equation}
\mathcal{G}=g_{i\bar{j}}dz^i\otimes d\bar{z}^j+g_{i\bar{j}}\delta \eta
^i\otimes \delta \bar{\eta}^j.  \label{1.2}
\end{equation}

A Hermitian connection $D$, of $(1,0)-$ type, which satisfies in addition $%
D_{JX}Y=JD_XY,$ for all $X$ horizontal vectors and $J$ the natural complex
structure of the manifold, is the so called Chern-Finsler connection (cf.
\cite{A-P}), in brief $C-F.$ The $C-F$ connection is locally given by the
following coefficients (cf. \cite{Mub}):
\begin{equation}
N_j^k=g^{\overline{m}k}\frac{\partial g_{l\overline{m}}}{\partial z^j}\eta
^l=L_{lj}^k\eta ^l\;;\;L_{jk}^i=g^{\overline{l}i}\delta _kg_{j\overline{l}%
}\;\;;\;C_{jk}^i=g^{\overline{l}i}\dot{\partial}_kg_{j\overline{l}}\;\;;\;L_{%
\overline{j}k}^{\overline{\imath }}=C_{\overline{j}k}^{\overline{\imath }}=0,
\label{1.3}
\end{equation}
where here and further on $\delta _k$ is the adapted frame of the $C-F$ $%
(c.n.c.)$ and $D_{\delta _k}\delta _j=L_{jk}^i\delta _i,$ $D_{\dot{\partial}%
_k}\dot{\partial}_j=C_{jk}^i\dot{\partial}_i,$ etc. The $C-F$ connection is
the main tool in this study.

Denoting by $"\shortmid "$ , $"\mid "$ , $"\bar{\shortmid}"$ and $"\bar{\mid}%
",$ the $h-$, $v-$, $\overline{h}-$, $\overline{v}-$ covariant derivatives
with respect to $C-F$ connection, respectively, for any $X^i$ it results
\begin{eqnarray}
X_{|k}^i &:&=\delta _kX^i+X^lL_{lk}^i\;;\;X^i|_k:=\stackrel{.}{\partial }%
_kX^i+X^lC_{lk}^i;  \label{II.11} \\
X_{|\overline{k}}^i &:&=\delta _{\overline{k}}X^i\;;\;X^i|_{\overline{k}}:=%
\stackrel{.}{\partial }_{\overline{k}}X^i;  \nonumber
\end{eqnarray}
and
\begin{eqnarray}
\eta _{|k}^i &=&\eta _{|\overline{k}}^i=\eta ^i|_{\overline{k}}=0\;;\;\;\eta
^i|_k=\delta _k^i;  \label{1.3.} \\
g_{i\overline{j}|k} &=&g_{i\overline{j}|\overline{k}}=g_{i\overline{j}%
}|_k=g_{i\overline{j}}|_{\overline{k}}=0;  \nonumber \\
(g_{i\overline{j}}\bar{\eta}^j)_{|k} &=&(g_{i\overline{j}}\bar{\eta}^j)_{|%
\bar{k}}=(g_{i\overline{j}}\bar{\eta}^j)|_k=0\;;\;\;(g_{i\overline{j}}\bar{%
\eta}^j)|_{\bar{k}}=g_{i\bar{k}.}  \nonumber
\end{eqnarray}

The nonzero curvatures of the $C-F$ connection are denoted by
\begin{eqnarray*}
R(\delta _h,\delta _{\bar{k}})\delta _j &=&R_{j\bar{k}h}^i\delta _i\;;\;R(%
\dot{\partial}_h,\delta _{\bar{k}})\delta _j=\Xi _{j\bar{k}h}^i\delta
_i\;;\;R(\delta _h,\dot{\partial}_{\bar{k}})\delta _j=P_{j\bar{k}h}^i\delta
_i \\
R(\delta _h,\delta _{\bar{k}})\dot{\partial}_j &=&R_{j\bar{k}h}^i\dot{%
\partial}_i\;;\;R(\dot{\partial}_h,\delta _{\bar{k}})\dot{\partial}_j=\Xi _{j%
\bar{k}h}^i\dot{\partial}_i\;;\;R(\delta _h,\dot{\partial}_{\bar{k}})\dot{%
\partial}_j=P_{j\bar{k}h}^i\dot{\partial}_i \\
R(\dot{\partial}_h,\dot{\partial}_{\bar{k}})\delta _j &=&S_{j\bar{k}%
h}^i\delta _i\;;\;\;R(\dot{\partial}_h,\dot{\partial}_{\bar{k}})\dot{\partial%
}_j=S_{j\bar{k}h}^i\dot{\partial}_i\;,
\end{eqnarray*}
where
\begin{eqnarray}
R_{j\overline{h}k}^i &=&-\delta _{\overline{h}}L_{jk}^i-\delta _{\overline{h}%
}(N_k^l)C_{jl}^i\;\;;\;\Xi _{j\overline{h}k}^i=-\delta _{\overline{h}%
}C_{jk}^i=\Xi _{k\overline{h}j}^i;\;  \label{1.4'} \\
P_{j\overline{h}k}^i &=&-\dot{\partial}_{\overline{h}}L_{jk}^i-\dot{\partial}%
_{\overline{h}}(N_k^l)C_{jl}^i\;\;;\;S_{j\overline{h}k}^i=-\dot{\partial}_{%
\overline{h}}C_{jk}^i=S_{k\overline{h}j}^i.  \nonumber
\end{eqnarray}

Considering the Riemann tensor
\begin{eqnarray}
\mathbf{R}(W,\overline{Z},X,\overline{Y}) &:&=G(R(X,\overline{Y})W,\overline{%
Z}),  \label{1.4''} \\
\mathbf{R}(W,\overline{Z},X,\overline{Y}) &=&\overline{\mathbf{R}(Z,%
\overline{W},Y,\overline{X})}  \nonumber
\end{eqnarray}
for $W,X$, $\overline{Z},\overline{Y}$ horizontal or vertical vectors, it
results the $h\bar{h}-,$ $h\bar{v}-,$ $v\bar{h}-,$ $v\bar{v}-$ Riemann type
tensors: $R_{\bar{j}i\bar{h}k}=g_{l\bar{j}}R_{i\bar{h}k}^l;$ $P_{\bar{j}i%
\bar{h}k}=g_{l\bar{j}}P_{i\bar{h}k}^l;$ $\Xi _{\bar{j}i\bar{h}k}=g_{l\bar{j}%
}\Xi _{i\bar{h}k}^l;$ $S_{\bar{j}i\bar{h}k}=g_{l\bar{j}}S_{i\bar{h}k}^l$,
which have properties $R_{i\overline{j}k\overline{h}}=R_{\overline{j}i%
\overline{h}k}\;\;;\;\Xi _{i\overline{j}k\overline{h}}=P_{\overline{j}i%
\overline{h}k};\;P_{i\overline{j}k\overline{h}}=\Xi _{\overline{j}i\overline{%
h}k}\;\;;\;S_{i\overline{j}k\overline{h}}=S_{\overline{j}i\overline{h}k}=S_{%
\overline{h}i\overline{j}k},$ where $R_{i\overline{j}k\overline{h}}:=%
\overline{R_{\bar{\imath}j\bar{k}h}},$ etc., (see \cite{Mub}, p. 77).

Further on, everywhere the index $0$ means the contraction by $\eta ,$ for
example $R_{0\overline{h}k}^i:=R_{j\overline{h}k}^i\eta ^j$.

\begin{proposition}
i) $R_{0\overline{h}k}^i=-\delta _{\overline{h}}N_k^i\;;\;\;$ $R_{\bar{r}0%
\overline{h}k}=-g_{i\bar{r}}\delta _{\overline{h}}N_k^i;$

ii) $P_{0\overline{h}k}^i=-g^{\bar{m}i}C_{0\bar{m}\bar{h}|k}\;;\;\;P_{\bar{r}%
0\overline{h}k}=-C_{0\bar{r}\bar{h}|k}\;;\;P_{0\overline{0}k}^i=0;$

iii) $\Xi _{j\overline{h}k}^i=-C_{jk|\bar{h}}^i\;;\;$ $S_{j\overline{h}%
k}^i=-C_{jk}^i|_{\bar{h}}\;;\;\Xi _{0\overline{h}k}^i=\Xi _{k\overline{h}%
0}^i=S_{0\overline{h}k}^i=S_{k\overline{h}0}^i\;;$

$\Xi _{\bar{r}j\overline{h}k}=-C_{j\bar{r}k|\bar{h}}\;;\;$ $S_{\bar{r}j%
\overline{h}k}=-C_{j\bar{r}k}|_{\bar{h}},$ where we denoted $C_{j\bar{r}%
k}:=C_{jk}^ig_{i\bar{r}}$ and $C_{r\bar{j}\bar{k}}$ is its conjugate$;$

iv) $C_{l\bar{r}\bar{h}|k}=(\dot{\partial}_{\bar{h}}L_{lk}^i)g_{i\bar{r}}+(%
\dot{\partial}_{\bar{h}}N_k^i)C_{i\bar{r}l};$

v) $C_{l\bar{r}h|k}=(\dot{\partial}_hL_{lk}^i)g_{i\bar{r}};$

vi) $P_{j\overline{h}k}^i-P_{0\overline{h}k}^i|_j-P_{0\overline{h}%
r}^iC_{kj}^r=0.$
\end{proposition}

\begin{proof}
i) and iii) results by (\ref{II.11}), (\ref{1.3.}), (\ref{1.4'}) and $C_{0k}^i=C_{k0}^i=0$.

For ii) we have

$P_{\bar{r}0\overline{h}k}=g_{i\bar{r}}P_{0\overline{h}k}^i=g_{i\bar{r}}\dot{%
\partial}_{\bar{h}}N_k^i=-g_{i\bar{r}}\dot{\partial}_{\bar{h}}\left( g^{\bar{%
m}i}\frac{\partial g_{j\bar{m}}}{\partial z^k}\eta ^j\right) $

$=g_{i\bar{r}}g^{\bar{m}l}g^{\bar{s}i}\left( \dot{\partial}_{\bar{h}}g_{l%
\bar{s}}\right) \frac{\partial g_{j\bar{m}}}{\partial z^k}\eta ^j-g_{i\bar{r}%
}g^{\bar{m}i}\dot{\partial}_{\bar{h}}\left( \frac{\partial g_{j\bar{m}}}{%
\partial z^k}\right) \eta ^j$

$=g^{\bar{m}l}\left( \dot{\partial}_{\bar{h}}g_{l\bar{r}}\right) \frac{%
\partial g_{j\bar{m}}}{\partial z^k}\eta ^j-\frac \partial {\partial
z^k}\left( \dot{\partial}_{\bar{h}}g_{j\bar{r}}\right) \eta ^j=C_{l\bar{r}%
\bar{h}}N_k^l-\frac \partial {\partial z^k}\left( C_{j\bar{r}\bar{h}}\eta
^j\right) .$

Because $C_{0\overline{rh}}:=C_{l\overline{rh}}\eta ^l$ it leads to

$C_{0\overline{rh}|k}=(C_{l\overline{rh}}\eta ^l)_{|k}=\delta _k(C_{l%
\overline{rh}}\eta ^l)=\frac \partial {\partial z^k}\left( C_{l\overline{rh}%
}\eta ^l\right) -N_k^s\dot{\partial}_s\left( (\dot{\partial}_{\bar{h}}g_{l%
\bar{r}})\eta ^l\right) $

$=\frac \partial {\partial z^k}\left( C_{l\overline{rh}}\eta ^l\right) -N_k^s%
\dot{\partial}_{\bar{h}}\left( (\dot{\partial}_sg_{l\bar{r}})\eta ^l\right)
-N_k^sC_{l\overline{rh}}\delta _s^l=\frac \partial {\partial z^k}\left( C_{l%
\overline{rh}}\eta ^l\right) -N_k^sC_{s\overline{rh}}.$ From here, result
the second relation of ii).$\ $The others immediately result by this.

Now, differentiating $N_k^ig_{i\bar{r}}=\frac{\partial g_{j\bar{r}}}{\partial z^k}\eta ^j$ with respect to $\eta ^l$ yields $L_{lk}^ig_{i\bar{r}}=\frac{\partial g_{l\bar{r}}}{\partial z^k}-N_k^iC_{i\bar{r}l},$ which differentiated by $\bar{\eta}^h$ leads to iv).

Differentiating $L_{lk}^ig_{i\bar{r}}=\frac{\partial g_{l\bar{r}}}{\partial z^k}-N_k^iC_{i\bar{r}l},$ by $\bar{\eta}^h$ it results v).

It is obvious that $P_{0\overline{h}k}^i=-\dot{\partial}_{\overline{h}}N_k^i$. Hence,

$P_{j\overline{h}k}^i=-\dot{\partial}_{\overline{h}}L_{jk}^i-\dot{\partial}_{%
\overline{h}}(N_k^l)C_{jl}^i=-\dot{\partial}_{\overline{h}}(\dot{\partial}%
_jN_k^i)+P_{0\overline{h}k}^lC_{jl}^i$

$=-\dot{\partial}_j(\dot{\partial}_{\overline{h}}N_k^i)+P_{0\overline{h}%
k}^lC_{jl}^i=\dot{\partial}_jP_{0\overline{h}k}^i+P_{0\overline{h}%
k}^lC_{jl}^i$

$=P_{0\overline{h}k}^i|_j+P_{0\overline{h}r}^iC_{kj}^r,$ i.e. vi).
\end{proof}

\begin{proposition}
For any $X\in \Gamma ^0(T^{\prime }M)$ the following properties hold true:

i) $X|_{k|j}-X_{|j}|_k=C_{jk}^iX_{|i};$

ii) $X|_{\bar{k}|j}-X_{|j}|_{\bar{k}}=-P_{0\bar{k}j}^iX|_i.$
\end{proposition}

\begin{proof}
We have

$\left[ \delta _j,\dot{\partial}_k\right] X=L_{kj}^i\left( \dot{\partial}%
_iX\right) =L_{kj}^iX|_i$ and

$\left[ \delta _j,\dot{\partial}_{\bar{k}}\right] X=-P_{0\bar{k}j}^i\dot{%
\partial}_iX=-P_{0\bar{k}j}^iX|_i.$

On the other hand,

$\left[ \delta _j,\dot{\partial}_k\right] X=\delta _j\left( \dot{\partial}%
_kX\right) -\dot{\partial}_k(\delta _jX)=\delta _j\left( X|_k\right) -\dot{%
\partial}_k(X_{|j})$

$=X|_{k|j}+L_{kj}^iX|_i-X_{|j}|_k-C_{jk}^iX_{|i}$ and

$\left[ \delta _j,\dot{\partial}_{\bar{k}}\right] X=\delta _j\left( \dot{%
\partial}_{\bar{k}}X\right) -\dot{\partial}_{\bar{k}}(\delta _jX)=\delta
_j\left( X|_{\bar{k}}\right) -\dot{\partial}_{\bar{k}}(X_{|j})$

$=X|_{\bar{k}|j}-X_{|j}|_{\bar{k}}.$

From the above relations it results i) and ii).
\end{proof}

For the vertical section $\mathcal{L}=\eta ^k\dot{\partial}_k,$ called the
\textit{Liouville} complex field (or the vertical radial vector field in
\cite{A-P}), we consider its horizontal lift $\chi :=\eta ^k\delta _k.$

According to \cite{A-P}, p. 108, \cite{Mub}, p. 81, the horizontal
holomorphic curvature of the complex Finsler space $(M,F)$ in direction $%
\eta $ is
\begin{equation}
K_F(z,\eta )=\frac{2\mathcal{G}(\mathbf{R}(\chi ,\bar{\chi})\chi ,\bar{\chi})%
}{[\mathbf{G}(\chi ,\bar{\chi})]^2}=\frac 2{L^2}\mathcal{G}(\mathbf{R}(\chi ,%
\bar{\chi})\chi ,\bar{\chi}).  \label{1.8}
\end{equation}

Let us recall that in \cite{A-P}'s terminology, the complex Finsler space $%
(M,F)$ is \textit{strongly K\"{a}hler} iff $T_{jk}^i=0,$ \textit{K\"{a}hler}$%
\;$iff $T_{jk}^i\eta ^j=0$ and \textit{weakly K\"{a}hler }iff\textit{\ } $%
g_{i\overline{l}}T_{jk}^i\eta ^j\overline{\eta }^l=0,$ where $%
T_{jk}^i:=L_{jk}^i-L_{kj}^i.$ In \cite{C-S} it is proved that strongly
K\"{a}hler and K\"{a}hler notions actually coincide. We notice that in the
particular case of complex Finsler metrics which come from Hermitian metrics
on $M,$ so-called \textit{purely Hermitian metrics} in \cite{Mub}, (i.e. $%
g_{i\overline{j}}=g_{i\overline{j}}(z)$)$,$ all those nuances of K\"{a}hler
coincide.

It is well known by \cite{A-P,Mub} that the complex geodesics curves are
defined by means of Chern-Finsler $(c.n.c.).$ But this $(c.n.c.)$ derives
from a complex spray if the complex metric only is weakly K\"{a}hler. On the
other hand, its local coefficients $N_j^k=g^{\bar{m}k}\frac{\partial g_{l%
\bar{m}}}{\partial z^j}\eta ^l$, always determine a complex spray with
coefficients $G^i=\frac 12N_j^i\eta ^j.$ Further on, $G^i$ induce a $%
(c.n.c.) $ by $\stackrel{c}{N_j^i}:=\dot{\partial}_jG^i$ called \textit{%
canonical} in \cite{Mub}, where it is proved that it coincides with
Chern-Finsler $(c.n.c.)$ if and only if the complex Finsler metric is
K\"{a}hler. Using canonical $(c.n.c.)$ we associate to it the next complex
linear connections: one of Berwald type
\[
B\Gamma :=\left( \stackrel{c}{N_j^i},\stackrel{B}{L_{jk}^i}:=\dot{\partial}_k%
\stackrel{c}{N_j^i}=\stackrel{B}{L_{kj}^i},\stackrel{B}{L_{j\bar{k}}^i}:=%
\dot{\partial}_{\bar{k}}\stackrel{c}{N_j^i},0,0\right)
\]
and another of Rund type
\[
R\Gamma :=\left( \stackrel{c}{N_j^i},\stackrel{c}{L_{jk}^i}:=\frac 12g^{%
\overline{l}i}(\stackrel{c}{\delta _k}g_{j\overline{l}}+\stackrel{c}{\delta
_j}g_{k\overline{l}}),\stackrel{c}{L_{j\bar{k}}^i}:=\frac 12g^{\overline{l}%
i}(\stackrel{c}{\delta _{\bar{k}}}g_{j\overline{l}}-\stackrel{c}{\delta _{%
\bar{l}}}g_{j\overline{k}}),0,0\right) ,
\]
where $\stackrel{c}{\delta _k}:=\frac \partial {\partial z^k}-\stackrel{c}{%
N_k^j}\dot{\partial}_j.$ $R\Gamma $ is only $h-$ metrical and $B\Gamma $ is
neither $h-$ nor $v-$ metrical, (for more details see \cite{Mub}). Note that
$2G^i=N_j^i\eta ^j=\stackrel{c}{N_j^i}\eta ^j=\stackrel{B}{L_{jk}^i}\eta
^j\eta ^k.$ Moreover, in the K\"{a}hler case we have $\stackrel{c}{\delta _k}%
=\delta _k$ and so, $L_{jk}^i=\stackrel{c}{L_{jk}^i}=\stackrel{B}{L_{jk}^i}$
and $\stackrel{c}{L_{j\bar{k}}^i}=0.$

Further on, everywhere in this paper the Berwald and Rund connections will
be specified by a super-index, like above (e.g. $\stackrel{c}{\delta _k}$, $%
\stackrel{B}{L_{jk}^i}$, $\stackrel{c}{L_{jk}^i}$, $X_{\stackrel{B}{|}k}$,
etc.), while for the Chern-Finsler connection will be kept the initial
generic notation without super-index (e.g. $\delta _k$, $L_{jk}^i$, $%
X_{|\;k} $, etc.).

In the real case, a Finsler space is Landsberg if the Berwald and Rund
connections coincide. Nevertheless, in complex Finsler geometry some
differences appear. We speak about complex \textit{Landsberg} space iff $%
\stackrel{B}{L_{jk}^i}=\stackrel{c}{L_{jk}^i}$, and about $G$\textit{\ -
Landsberg} space iff $\stackrel{B}{L_{jk}^i}=\stackrel{c}{L_{jk}^i}$ and the
spray coefficients are holomorphic functions with respect to $\eta $, i.e. $%
\dot{\partial}_{\bar{k}}G^i=0$, (see \cite{Al-Mu1}).

\begin{theorem}
(\cite{Al-Mu1}) Let $(M,F)$ be a $n$ - dimensional complex Finsler space.
Then the following assertions are equivalent:

i) $(M,F)$ is a $G$ - Landsberg space;

ii) $\stackrel{B}{L_{jk}^i}=\stackrel{c}{L_{jk}^i}(z);$

iii) $C_{l\bar{r}h\stackrel{B}{|}0}=0$ and $C_{r\bar{0}h\stackrel{B}{|}\bar{0%
}}=0,$ where $\stackrel{B}{\shortmid }$ is $h-$covariant derivative with
respect to $B\Gamma $ connection.
\end{theorem}

We note that any complex Finsler space which is K\"{a}hler is Landsberg,
too. So, by replacing the Landsberg condition from definition of the $G$ -
Landsberg space with the K\"{a}hler condition, we have obtained another
class of complex Finsler spaces, called us $G$\textit{- K\"{a}hler}. On the
other hand, keeping with Aikou's work, a complex Finsler space which is
K\"{a}hler and $L_{jk}^i=L_{jk}^i(z)$ is named complex \textit{Berwald}
space. Some tensorial characterizations for these classes of complex Finsler
spaces are contained in the next theorem.

\begin{theorem}
(\cite{Al-Mu1}) Let $(M,F)$ be a $n$ - dimensional complex Finsler space.
Then the following assertions are equivalent:

i) $(M,F)$ is $G$ - K\"{a}hler;

ii) $\stackrel{B}{L_{j\bar{k}}^i}=\stackrel{c}{L_{j\bar{k}}^i}$;

iii) $(M,F)$ is a complex Berwald space;

iv) $(M,F)$ is a K\"{a}hler and either $C_{l\bar{r}h|\bar{k}}=0$ or $C_{l%
\bar{r}h|k}=0$.
\end{theorem}

From Proposition 2.1 iii) and by $\Xi _{i\overline{j}k\overline{h}}=P_{%
\overline{j}i\overline{h}k}$, a complex Berwald space is a K\"{a}hler space
with either $\Xi _{i\overline{j}k\overline{h}}=0$ or $P_{\overline{j}i%
\overline{h}k}=0.$ Between above classes of complex Finsler spaces we have
the inclusions: complex Berwald space $\subset $ $G$ - Landsberg $\subset $
complex Landsberg space.

\section{The complex Berwald frame}

\setcounter{equation}{0}Let $(M,F)$ be a $2$ - dimensional complex Finsler
space, $(z^k,\eta ^k)_{k=\overline{1,2}}$ be complex coordinates on $%
T^{\prime }M$ and $VT^{\prime }M$ be the vertical bundle spanned by $\{\dot{%
\partial}_k\}.$ Further on, the indices $i,j,k,...$ run over $\{1,2\}.$ Let $%
g_{i\bar{j}}$ be the fundamental metric tensor of the space and $\mathcal{G}$
the Hermitian metric structure (\ref{1.2}), defined on $T_C(T^{\prime }M),$
with respect to the adapted frames of Chern-Finsler $(c.n.c.).$

We set $l:=l^i\dot{\partial}_i$ and its dual form is $\omega =l_i\delta \eta
^i,$ where
\begin{equation}
l^i=\frac 1F\eta ^i\;\;\mbox{and}\;\;l_i=\frac 1Fg_{i\bar{j}}\bar{\eta}%
^j=g_{i\bar{j}}l^{\bar{j}}.  \label{2.1}
\end{equation}

Now, our aim is to construct an orthonormal frame in the vertical bundle $%
VT^{\prime }M$, which is $2$ - dimensional in any point. Therefore, it is
decomposed into $VT^{\prime }M=\{l\}\oplus \{l\}^{\perp },$ where $%
\{l\}^{\perp }$ is spanned by a complex vector $m$. Requiring the
orthogonality condition $\mathcal{G}(l,\bar{m})=0$ and $\mathcal{G}(m,\bar{m}%
)=1$, i.e. $m$ is a unit vector and, using $m_i:=g_{i\bar{j}}m^{\bar{j}},$
the above two conditions get the linear system $\left\{
\begin{array}{c}
l_1m^1+l_2m^2=0 \\
m_1m^1+m_2m^2=1
\end{array}
\right. .$

We try to solve this system following the same technique from \cite{Be} for
real case. Nevertheless, let us pay more attention to this system. Passing
in real coordinates, it contains three real equations with four real
unknowns. So that it doesn't admit an unique solution. Formally, solving
this system as one linear, it is obtained the 'solutions' $m^1=\frac{-l_2}%
\Delta $, $m^2=\frac{l_1}\Delta ,$ $m_1=-\Delta l^2$ and $m_2=\Delta l^1,$
where $\Delta =l_1m_2-l_2m_1,$ which indeed are not completely determined
because $\Delta $ depends on $m_i.$ We can say more about these 'solutions'.
A straightforward computation proves that $|\Delta |=\sqrt{g}$ and $\Delta
^{\prime }=\mathcal{T}\Delta $ under a change of the local coordinates $%
(z^k,\eta ^k)_{k=\overline{1,2}}$ into $(z^{\prime k},\eta ^{\prime k})_{k=%
\overline{1,2}}$, where $g:=\det (g_{i\bar{j}})$ and $\mathcal{T}:=\det
\left( \frac{\partial z^i}{\partial z^{\prime j}}\right) $. Therefore, a
natural question is if there exists at least $\Delta $ with the above
mentioned properties. The answer will come below, when we find two distinct
particular solutions for $\Delta .$

Subsequently, our statement will be made for a fixed choice of $\Delta $ and
then $\{l,m,\bar{l},\bar{m}\}$ with
\begin{equation}
m=\frac 1\Delta (-l_2\dot{\partial}_1+l_1\dot{\partial}_2)  \label{2.1'}
\end{equation}
will be called the \textit{complex Berwald frame}. Surely, the dependence of
the chosen for $\Delta $ will be analyzed everywhere.

But when we work in a fixed local chart, we can choose $\Delta =\sqrt{g},$
i.e. $\Delta $ is real, which produces the unique solutions $m^1=\frac{-l_2}{%
\sqrt{g}},$ $m^2=\frac{l_1}{\sqrt{g}},$ $m_1=-\sqrt{g}l^2$ and $m_2=\sqrt{g}%
l^1.$ Thus, we have
\begin{equation}
m=\frac 1{\sqrt{g}}(-l_2\dot{\partial}_1+l_1\dot{\partial}_2),  \label{2.2}
\end{equation}
in this fixed chart.

Then $\{l,m,\bar{l},\bar{m}\}$, with $m$ given by (\ref{2.2}) will be called
the \textit{local complex Berwald frame} of the space.

Note that (\ref{2.2}) provides only a local frame, because the set of
natural local basis in every chart does not have tensorial character. For
this reason, considering a change of the local coordinates, we obtain
\[
m^{\prime }=\frac{\mathcal{T}}{|\mathcal{T}|}m\;;\;m^{^{\prime }i}=\frac{%
\mathcal{T}}{|\mathcal{T}|}\frac{\partial z^{\prime i}}{\partial z^k}%
m^k\;;\;m_i^{\prime }=\frac{\overline{\mathcal{T}}}{|\mathcal{T}|}\frac{%
\partial z^r}{\partial z^{\prime i}}m_r,
\]
which show that $m$ is not a vector, but it depends on the local change.
Therefore, it will say that $m$ from (\ref{2.2}) is a pseudo-vector.

Although $m$ from (\ref{2.2}) depends on the local changes of the
coordinates, it is very important in our study, in a fixed chart. Certainly,
further on we will be very careful with the global validity of our
assertions. We will see that together with its horizontal extension it gives
rise to some invariants which will characterize two dimensional complex
Finsler spaces. A first and useful remark is that the quantities $m_im^j$, $%
m^im^{\bar{j}},$ $m_im_{\bar{j}}$ and $m_im$ are independent of the chosen
local chart, and hence they have global meaning.

With respect to the local complex Berwald frame, $\dot{\partial}_k$ and $g_{i%
\bar{j}}$ are decomposed as follows
\begin{equation}
\dot{\partial}_i=l_il+m_im\;\;\;\mbox{and hence}\;\;\;g_{i\bar{j}}=l_il_{%
\bar{j}}+m_im_{\bar{j}}.  \label{2.3}
\end{equation}
From here we deduce that
\begin{equation}
C_{jk}^i=g^{\bar{m}i}\dot{\partial}_kg_{j\bar{m}}=Al^im_km_j+Bm^im_km_j,
\label{2.4}
\end{equation}
where we set
\[
A:=m^jm^kl_hC_{kj}^h\;\;;\;\;B:=m_hm^km^jC_{jk}^h.
\]

The dependence of the vertical terms $A$ and $B$ of the local charts is
obvious, $A^{\prime }=\frac{\mathcal{T}^2}{|\mathcal{T}|^2}A\;;\;B^{\prime }=%
\frac{\mathcal{T}}{|\mathcal{T}|}B$. Thus, $A$ and $B$ are not invariants,
but if they are zero in a local chart, then they are zero in any local
chart. Moreover, by means of $A$ and $B$ and setting $\Delta =B\sqrt{g}$
with $|B|^2=1$ or $\Delta =\sqrt{Ag}$ with $|A|^2=1,$ we obtain two
particular solutions for $m$ from (\ref{2.1'}) which certify the existence
of the complex Berwald frames.

Further on, all our work will be with respect to the local complex Berwald
frame, where $m$ is given by (\ref{2.2}).

Therefore, the formulas from Proposition 3.2, in \cite{Mu-Al}, become
\begin{eqnarray}
l(l_{i}) &=&\frac{-1}{2F}l_{i}\;;\;\bar{l}(l_{i})=\frac{1}{2F}%
l_{i}\;;\;\;l(m_{i})=\frac{1}{2F}m_{i}\;;\;\bar{l}(m_{i})=\frac{-1}{2F}m_{i};
\label{2.4'} \\
m(l_{i}) &=&Am_{i}\;;\;\;\bar{m}(l_{i})=\frac{1}{F}m_{i}\;;\;\;m(m_{i})=%
\frac{1}{2}Bm_{i}-\frac{1}{F}l_{i};\;\bar{m}(m_{i})=\frac{1}{2}\bar{B}m_{i};
\nonumber \\
l(l^{i}) &=&\frac{1}{2F}l^{i}\;;\;\;\bar{l}(l^{i})=-\frac{1}{2F}%
l^{i}\;;\;\;l(m^{i})=-\frac{1}{2F}m^{i}\;\;;\;\;\bar{l}(m^{i})=\frac{1}{2F}%
m^{i};  \nonumber \\
m(l^{i}) &=&\frac{1}{F}m^{i}\;;\;\;\bar{m}(l^{i})=0\;;\;  \nonumber \\
m(m^{i}) &=&-\frac{1}{2}Bm^{i}-Al^{i}\;\;;\;\;\bar{m}(m^{i})=-\frac{1}{F}%
l^{i}-\frac{1}{2}\bar{B}m^{i}.  \nonumber
\end{eqnarray}

By a direct computation, using the above relations, we obtain formulas for
the vertical covariant derivatives of $l,m,\bar{l}$ and $\bar{m}$ with
respect to the $C-F$ connection
\begin{eqnarray}
l_i|_j &=&\frac{-1}{2F}l_il_j\;+Am_im_j;\;\;l_i|_{\bar{j}}=\frac 1{2F}l_il_{%
\bar{j}}+\frac 1Fm_im_{\bar{j}};  \label{2.4''} \\
m_i|_j &=&\frac 1{2F}m_il_j-\frac 1Fl_im_j-\frac B2m_im_j\;;\;\;m_i|_{\bar{j}%
}=\frac{-1}{2F}m_il_{\bar{j}}+\frac{\bar{B}}2m_im_{\bar{j}};  \nonumber \\
l^i|_j &=&\frac 1F\delta _j^i-\frac 1{2F}l_jl^i\;;\;\;l^i|_{\bar{j}}=\frac{-1%
}{2F}l_{\bar{j}}l^i\;;\;F|_j=\frac 12l_j;\;\;  \nonumber \\
m^i|_j &=&\frac{-1}{2F}l_jm^i+\frac B2m_jm^i\;;\;\;m^i|_{\bar{j}}=\frac
1{2F}l_{\bar{j}}m^i-\frac 1Fm_{\bar{j}}l^i-\frac{\bar{B}}2m_{\bar{j}}m^i,
\nonumber
\end{eqnarray}
and their conjugates.

Moreover, because $\bar{l}(C_{kj}^h)=0$ and $l(C_{kj}^h)=-\frac 1FC_{kj}^h$
by some computation, it results
\begin{eqnarray}
A|_{\bar{h}} &=&\dot{\partial}_{\bar{h}}A=(l_{\bar{h}}\bar{l}+m_{\bar{h}}%
\bar{m})A=\frac{3A}{2F}l_{\bar{h}}+A|_{\bar{s}}m^{\bar{s}}m_{\bar{h}};
\label{2.4'''} \\
B|_{\bar{h}} &=&\dot{\partial}_{\bar{h}}B=(l_{\bar{h}}\bar{l}+m_{\bar{h}}%
\bar{m})B=\frac B{2F}l_{\bar{h}}+B|_{\bar{s}}m^{\bar{s}}m_{\bar{h}};
\nonumber \\
A|_h &=&\dot{\partial}_hA=\left( l_hl+m_hm\right) A=-\frac{5A}{2F}%
l_h+A|_sm^sm_h;  \nonumber \\
B|_h &=&\dot{\partial}_hB=\left( l_hl+m_hm\right) B=-\frac{3B}{2F}%
l_h+B|_sm^sm_h.  \nonumber
\end{eqnarray}

Now, via the natural isomorphism between the bundles $VT^{\prime }M$ and $%
T^{\prime }M$, composed with the horizontal lift of $HT^{\prime }M,$ we
obtain the following orthonormal local frame on $H_CT^{\prime }M,$%
\[
\{\lambda :=l^i\delta _i,\;\;\mu =m^i\delta _i,\;\;\bar{\lambda}:=l^{\bar{\imath}%
}\delta _{\bar{\imath}},\;\;\bar{\mu}=m^{\bar{\imath}}\delta _{\bar{\imath}}\}.
\]

Let $D$ be the $C-F$ connection on $(M,F).$ Further on, let us give an
explicit expression for $C-F$ connection with respect to horizontal local
frame $\{\lambda ,\mu ,\bar{\lambda},\bar{\mu}\}.$ Moreover, using (\ref{2.3}%
) and $L_{jk}^i=g^{\bar{m}i}\delta _kg_{j\bar{m}}$ it results
\begin{eqnarray}
L_{jk}^i &=&Jl^il_jl_k+Ul^im_jl_k+Vl^il_jm_k+Xl^im_jm_k  \label{2.10} \\
&&+Om^il_jl_k+Ym^im_jl_k+Em^il_jm_k+Hm^im_jm_k,  \nonumber
\end{eqnarray}
where we set
\begin{eqnarray}
J
&:&=l^jl^kl_iL_{jk}^i;\;U:=m^jl^kl_iL_{jk}^i;\;V:=l^jm^kl_iL_{jk}^i;%
\;X:=m^jm^kl_iL_{jk}^i  \label{2.11} \\
O
&:&=l^jl^km_iL_{jk}^i;\;Y:=m^jl^km_iL_{jk}^i;\;E:=l^jm^km_iL_{jk}^i;%
\;H:=m^jm^km_iL_{jk}^i.  \nonumber
\end{eqnarray}

Here the horizontal settled quantities do not have tensorial character,
because under the change of charts we have
\begin{eqnarray}
J^{\prime } &=&J+\mathcal{T}_{ab}^rl^al^bl_r\;;\;U^{\prime }=\frac{\mathcal{T%
}}{|\mathcal{T}|}(U+\mathcal{T}_{ab}^rm^al^bl_r)\;;\;  \label{c} \\
V^{\prime } &=&\frac{\mathcal{T}}{|\mathcal{T}|}(V+\mathcal{T}%
_{ab}^rl^am^bl_r)\;;\;X^{\prime }=\frac{\mathcal{T}^2}{|\mathcal{T}|^2}(X+%
\mathcal{T}_{ab}^rm^am^bl_r)\;;  \nonumber \\
O^{\prime } &=&\frac{\overline{\mathcal{T}}}{|\mathcal{T}|}(O+\mathcal{T}%
_{ab}^rl^al^bm_r)\;;\;Y^{\prime }=Y+\mathcal{T}_{ab}^rm^al^bm_r\;;\;
\nonumber \\
E^{\prime } &=&E+\mathcal{T}_{ab}^rl^am^bm_r\;;\;H^{\prime }=\frac{\mathcal{T%
}}{|\mathcal{T}|}(H+\mathcal{T}_{ab}^rm^am^bm_r),  \nonumber
\end{eqnarray}
where $\mathcal{T}_{ab}^r:=\frac{\partial z^{\prime j}}{\partial z^a}\frac{%
\partial z^{\prime k}}{\partial z^b}\frac{\partial ^2z^r}{\partial z^{\prime
j}\partial z^{\prime k}}$.

Firstly, the properties of the $C-F$ connection $N_k^i=L_{jk}^i\eta ^j$ and $%
\dot{\partial}_jN_k^i=L_{jk}^i,$ (see \cite{Mub}), permit us to establish
some links between the vertical and horizontal terms (\ref{2.11}) of this
connection. Indeed,

$N_k^i=F(Jl^il_k+Vl^im_k+Om^il_k+Em^im_k)$ and

$L_{jk}^i=(l_jl+m_jm)[F(Jl^il_k+Vl^im_k+Om^il_k+Em^im_k)]$

$=[\frac 12J+Fl(J)]l^il_jl_k+[Fm(J)-V-FAO]l^im_jl_k+[Fl(V)+\frac
32V]l^il_jm_k$

$+[Fm(V)+FAJ+\frac 12FBV-FAE]l^im_jm_k+[Fl(O)-\frac 12O]m^il_jl_k$

$+[Fm(O)+J-\frac 12FBO-E]m^im_jl_k+[Fl(E)+\frac 12E]m^il_jm_k$

$+[Fm(E)+V+FAO]m^im_jm_k$ which together with (\ref{2.10}) give,

\begin{proposition}
Let $(M,F)$ be a $2$ - dimensional complex Finsler space. Then

i) $J|_k=\frac 1{2F}Jl_k+[\frac 1F(U+V)+AO]m_k;$

ii) $V|_k=-\frac 1{2F}Vl_k+[A(E-J)-\frac 12BV+\frac 1FX]m_k;$

iii) $O|_k=\frac 3{2F}Ol_k+[\frac 1F(E+Y-J)+\frac 12BO]m_k;$

iv) $E|_k=\frac 1{2F}El_k+[\frac 1F(H-V)-AO]m_k.$
\end{proposition}

\begin{proof}
In the fixed local chart the assertions i)-iv) are true. We must
prove their global validity. For example, under the change of a local chart,
we have

$V^{\prime }|_k^{\prime }+\frac 1{2F}V^{\prime }l_k^{\prime }-[A^{\prime
}(E^{\prime }-J^{\prime })-\frac 12B^{\prime }V^{\prime }+\frac 1FX^{\prime
}]m_k^{\prime }$

$=\frac{\mathcal{T}}{|\mathcal{T}|}\frac{\partial z^r}{\partial z^{^{\prime
}k}}\{V|_r+\frac 1{2F}Vl_r-[A(E-J)-\frac 12BV+\frac 1FX]m_r\},$ where $%
V^{\prime }|_k^{\prime }:=\dot{\partial}_k^{\prime }V^{\prime }.$

Because $V|_r+\frac 1{2F}Vl_r-[A(E-J)-\frac 12BV+\frac 1FX]m_r=0,$ by its
change rule it results that it is zero in any local chart. Analogous results the geometric character of the others assertions.
\end{proof}

\begin{proposition}
Let $(M,F)$ be a $2$ - dimensional complex Finsler space. Then

i) It is K\"{a}hler if and only if $U=V$ and $Y=E;$

ii) It is weakly K\"{a}hler if and only if $U=V.$
\end{proposition}

\begin{proof}
i) By (\ref{2.10}), $%
L_{jk}^i-L_{kj}^i=(U-V)l^im_jl_k+(V-U)l^il_jm_k+(Y-E)m^im_jl_k+(E-Y)m^il_jm_k.
$ So, $L_{jk}^i-L_{kj}^i=0$ if and only if $U=V$ and $Y=E.$

To prove ii) we compute $g_{i\overline{l}}T_{jk}^i\eta ^j\overline{\eta }%
^l=F^2(L_{jk}^i-L_{kj}^i)l_il^j=F^2(V-U)m_k.$ It results $g_{i\overline{l}%
}T_{jk}^i\eta ^j\overline{\eta }^l=0$ if and only if $U=V$.

Taking into account the local changes of $U-V$ and $Y-E$, it follows the global validity of these statements.
\end{proof}

Further on, several calculus imply the following properties.

\begin{proposition}
With respect to the local Berwald frame, we have:
\begin{eqnarray}
\lambda (l_i) &=&Jl_i+Um_i\;;\;\bar{\lambda}(l_i)=\bar{\lambda}%
(l^i)=0\;;\;\;\lambda (l^i)=-Jl^i-Om^i\;;  \label{2.5} \\
\lambda (m_i) &=&Ol_i-\frac 12(J-Y)m_i\;;\;\bar{\lambda}(m_i)=\frac 12(\bar{J%
}+\bar{Y})m_i\;;\;  \nonumber \\
\lambda (m^i) &=&-Ul^i+\frac 12(J-Y)m^i\;;\;\bar{\lambda}(m^i)=-\frac 12(%
\bar{J}+\bar{Y})m^i\;;  \nonumber \\
\mu (l_i) &=&Vl_i+Xm_i\;;\;\;\bar{\mu}(l_i)=\bar{\mu}(l^i)=0\;;\;\mu
(l^i)=-Vl^i-Em^i\;;  \nonumber \\
\mu (m_i) &=&El_i+\frac 12(H-V)m_i\;;\;\bar{\mu}(m_i)=\frac 12(\bar{V}+\bar{H%
})m_i\;;\;  \nonumber \\
\mu (m^i) &=&-Xl^i-\frac 12(H-V)m^i\;;\;\bar{\mu}(m^i)=-\frac 12(\bar{V}+%
\bar{H})m^i\;;\;  \nonumber \\
\lambda (g) &=&(J+Y)g\;;\;\mu (g)=(V+H)g\;;\;\delta _i=l_i\lambda +m_i\mu
\;;\;\lambda (L)=\mu (L)=0  \nonumber
\end{eqnarray}
and their conjugates.
\end{proposition}

Then, from (\ref{2.5}) we deduce that
\begin{eqnarray}
l_{i|j} &=&l_{i|\bar{j}}=l_{|j}^i=l_{|\bar{j}}^i=0;  \label{2.5''} \\
m_{i|j} &=&-\frac 12[(J+Y)l_j+(V+H)m_j]m_i;\;m_{i|\bar{j}}=\frac 12[(\bar{J}+%
\bar{Y})l_{\bar{j}}+(\bar{V}+\bar{H})m_{\bar{j}}]m_i;\;  \nonumber \\
m_{|j}^i &=&\frac 12[(J+Y)l_j+(V+H)m_j]m^i;\;m_{|\bar{j}}^i=-\frac 12[(\bar{J%
}+\bar{Y})l_{\bar{j}}+(\bar{V}+\bar{H})m_{\bar{j}}]m^i  \nonumber
\end{eqnarray}
and theirs conjugates.

\section{Curvatures of the C-F connection}

\setcounter{equation}{0}In this section, we shall compute the curvature
coefficients of the $C-F$ connection with respect to the local frames $\{l,m,%
\bar{l},\bar{m}\}$ and $\{\lambda ,\mu ,\bar{\lambda},\bar{\mu}\}$. By means
of these, we characterize the $2$ - dimensional complex Finsler spaces.

\subsection{The $v\bar{v}-$ Riemann type tensor}

Firstly, we study the $v\bar{v}-$ Riemann type tensor $S_{\bar{r}j\overline{%
h}k}.$ Taking into account Proposition 2.1 iii) and the formulas (\ref{2.4}%
), (\ref{2.4''}) and (\ref{2.4'''}), we have

$S_{\bar{r}j\overline{h}k}=-(Al_{\bar{r}}m_jm_k+Bm_{\bar{r}}m_jm_k)|_{\bar{h}%
}$

$=[-A|_{\bar{h}}+\frac{3A}{2F}l_{\bar{h}}+(-A\bar{B}+\frac BF)m_{\bar{h}}]l_{%
\bar{r}}m_jm_k$

$+(-B|_{\bar{h}}+\frac B{2F}l_{\bar{h}}-\frac{B\bar{B}}2m_{\bar{h}})m_{\bar{r%
}}m_jm_k$

$=(-A|_{\bar{s}}m^{\bar{s}}-A\bar{B}+\frac BF)m_{\bar{h}}l_{\bar{r}%
}m_jm_k+(-B|_{\bar{s}}m^{\bar{s}}-\frac{B\bar{B}}2)m_{\bar{h}}m_{\bar{r}%
}m_jm_k.$

But, $S_{\bar{r}j\overline{h}k}$ is symmetric in $j,k$ and $\bar{r},\bar{h}.$
Therefore, it results that
\begin{eqnarray}
S_{\bar{r}j\overline{h}k} &=&\mathbf{I}m_{\bar{h}}m_{\bar{r}%
}m_jm_k\;;\;\;A|_{\bar{s}}m^{\bar{s}}=-A\bar{B}+\frac BF,  \label{2.5'} \\
\mbox{where}\;\;\;\;\mathbf{I}:= &&-B|_{\bar{s}}m^{\bar{s}}-\frac{B\bar{B}}2.
\nonumber
\end{eqnarray}

We note that $\mathbf{I}$ is invariable to the changes of the local
coordinates thanks to $S_{\bar{r}j\overline{h}k}$ and $m_jm_{\bar{h}}m_km_{%
\bar{r}}$ which are tensors. Further on, we point out some properties of the
function $\mathbf{I}$, called by us the \textit{vertical curvature invariant}%
.

By analogy with (\ref{1.8}), we define \textit{the vertical holomorphic
sectional curvature} in direction $l$%
\begin{equation}
K_{F,l}^v(z,\eta ):=2\mathbf{R(}l,\bar{l},l,\bar{l})  \label{2.8}
\end{equation}
and \textit{the vertical holomorphic sectional curvature} in direction $m$%
\begin{equation}
K_{F,m}^v(z,\eta ):=2\mathbf{R(}m,\bar{m},m,\bar{m})  \label{2.8'}
\end{equation}

\begin{theorem}
Let $(M,F)$ be a 2 - dimensional complex Finsler space. Then

i) $K_{F,l}^v(z,\eta )=0;$

ii) $K_{F,m}^v(z,\eta )=2\mathbf{I}$ and $\mathbf{I}$ is real valued;

iii) $\mathbf{I}|_0=-\mathbf{I}$.
\end{theorem}

\begin{proof}
By (\ref{2.5'}) $\mathbf{R(}l,\bar{l},l,\bar{l})=l^{\bar{h}}l^{\bar{r}%
}l^jl^kS_{\bar{r}j\overline{h}k}=0$ and \newline
$\mathbf{R(}m,\bar{m},m,\bar{m})=m^{\bar{h}}m^{\bar{r}}m^jm^kS_{\bar{r}j%
\overline{h}k}=\mathbf{I.}$ Indeed, $\mathbf{\bar{I}}=\overline{\mathbf{R(}m,%
\bar{m},m,\bar{m})}=\mathbf{R(}m,\bar{m},m,\bar{m})=\mathbf{I.}$ These imply
i) and ii).

Considering the Bianchi identity $S_{\bar{r}j\overline{h}k}|_i=S_{\bar{r}j%
\overline{h}i}|_k,$ (see \cite{Mub}, p. 77) and using the relations (\ref
{2.4}) and (\ref{2.5'}), we have

$\mathbf{I}|_im_{\bar{h}}m_{\bar{r}}m_jm_k-\frac 1F\mathbf{I}m_{\bar{h}}m_{%
\bar{r}}m_jm_il_k=\mathbf{I}|_km_{\bar{h}}m_{\bar{r}}m_jm_i-\frac 1F\mathbf{I%
}m_{\bar{h}}m_{\bar{r}}m_jm_kl_i,$ which contracted by $m^{\bar{h}}m^{\bar{r}%
}m^jm^il^k$ gives iii).
\end{proof}

\begin{proposition}
Let $(M,F)$ be a 2 - dimensional complex Finsler space.

i) It is purely Hermitian if and only if $A=0;$

ii) If $|A|\neq 0$ and $B=0$ then $\mathbf{I}=0$ and $A|_{\bar{h}}=\frac{3A}{%
2F}l_{\bar{h}}.$
\end{proposition}

\begin{proof}
By (\ref{2.5'}), $A=0$ implies $B=0.$ These give $C_{i\bar{h}j}=0$ which
means that $\frac{\partial g_{i\bar{h}}}{\partial \eta ^j}=0,$ i.e. $F$ is
purely Hermitian. Conversely, if $F$ is purely Hermitian then $A=B=0.$ Thus,
the assertion i) is proved.

The claim ii) follows readily from (\ref{2.5'}) and (\ref{2.4'''}).  Obviously, the statements are independent of the changes of local charts.
\end{proof}

The above Proposition shows that there are $2$ - dimensional complex Finsler
spaces with $K_{F,m}^v(z,\eta )=0$ which are not purely Hermitian.
Subsequently, we pay more attention to the case $AB^2\neq 0.$

\subsection{The $v\bar{h}-$ Riemann type tensor}

Let $\Xi _{\overline{r}j\overline{h}k}$ be the $v\bar{h}-$ Riemann type
tensor. Using the Proposition 2.1 iii) and the formulas (\ref{2.4}) and (\ref
{2.5''}), we have
\begin{eqnarray}
\Xi _{\overline{r}j\overline{h}k} &=&-[A_{|\bar{h}}l_{\bar{r}}+A(\bar{J}+%
\bar{Y})l_{\bar{r}}l_{\bar{h}}+A(\bar{V}+\bar{H})l_{\bar{r}}m_{\bar{h}}
\label{2.8'''} \\
&&+B_{|\bar{h}}m_{\bar{r}}+\frac B2(\bar{J}+\bar{Y})m_{\bar{r}}l_{\bar{h}%
}+\frac B2(\bar{V}+\bar{H})m_{\bar{r}}m_{\bar{h}}]m_jm_k.  \nonumber
\end{eqnarray}

We wish to investigate the relationship among $A$, $B,$ $\mathbf{I}$ and to
characterize the 2 - dimensional complex Finsler spaces by means of these.
For this, contracting the Bianchi identity
\begin{equation}
\Xi _{\overline{r}j\bar{h}k}|_{\bar{s}}-S_{\bar{r}j\bar{s}k|\bar{h}}+\Xi _{%
\bar{r}j\bar{p}k}\overline{C_{sh}^p}=0,  \label{2.9;}
\end{equation}
(see \cite{Mub}, p. 77), with the tensor $m^{\bar{r}}m^jm^km^{\bar{s}}$ and
taking into account (\ref{2.8'''}) and (\ref{2.4''}), we obtain

$\Xi _{\overline{r}j\overline{h}k}|_{\bar{s}}m^{\bar{r}}m^jm^km^{\bar{s}%
}=-\{B_{|\bar{h}}|_{\bar{s}}m^{\bar{s}}+\frac{\bar{B}}2B_{|\bar{h}}$

$+\frac 12[-\mathbf{I}(\bar{J}+\bar{Y})+B(\bar{J}+\bar{Y})|_{\bar{s}}m^{\bar{%
s}}-\frac BF(\bar{V}+\bar{H})]l_{\bar{h}}$

$+\frac 12[B|_{\bar{s}}m^{\bar{s}}(\bar{V}+\bar{H})+B(\bar{V}+\bar{H})|_{%
\bar{s}}m^{\bar{s}}]m_{\bar{h}}\};$

$S_{\bar{r}j\bar{s}k|\bar{h}}m^{\bar{r}}m^jm^{\bar{s}}m^k=\mathbf{I}_{|\bar{h%
}}$;

$\Xi _{\bar{r}j\bar{p}k}\overline{C_{sh}^p}m^{\bar{r}}m^jm^km^{\bar{s}}=-[%
\frac{\bar{A}}FB_{|\bar{0}}+\bar{B}B_{|\bar{p}}m^{\bar{p}}+\frac{\bar{A}B}2(%
\bar{J}+\bar{Y})+\frac{B\bar{B}}2(\bar{V}+\bar{H})]m_{\bar{h}}.$

Hence
\begin{eqnarray}
B_{|\bar{h}}|_{\bar{s}}m^{\bar{s}} &=&-\{\frac 12[-\mathbf{I}(\bar{J}+\bar{Y}%
)+B(\bar{J}+\bar{Y})|_{\bar{s}}m^{\bar{s}}-\frac BF(\bar{V}+\bar{H})]l_{\bar{%
h}}  \label{2.9''} \\
&&+\frac 12[(-\mathbf{I+}\frac{B\bar{B}}2)(\bar{V}+\bar{H})+B(\bar{V}+\bar{H}%
)|_{\bar{s}}m^{\bar{s}}  \nonumber \\
&&+2\frac{\bar{A}}FB_{|\bar{0}}+2\bar{B}B_{|\bar{p}}m^{\bar{p}}+\bar{A}B(%
\bar{J}+\bar{Y})]m_{\bar{h}}+\mathbf{I}_{|\bar{h}}+\frac{\bar{B}}2B_{|\bar{h}%
}\}  \nonumber
\end{eqnarray}
and its conjugate.

On the other hand, contracting in (\ref{2.9;}) by $m^{\bar{r}}m^jm^kl^{\bar{s%
}}$, using

$\Xi _{\bar{r}j\overline{h}k}|_{\bar{s}}m^{\bar{r}}m^jm^kl^{\bar{s}}=-\frac
1F\{B_{|\bar{h}}|_{\bar{0}}-\frac 12B_{|\bar{h}}+\frac B2[-\frac 12(\bar{J}+%
\bar{Y})+(J+Y)|_{\bar{0}}]l_{\bar{h}}$

$+\frac B2[\frac 12(\bar{V}+\bar{H})+(\bar{V}+\bar{H})|_{\bar{0}}]m_{\bar{h}%
}\}$ and

$S_{\bar{r}j\bar{s}k|\bar{h}}m^{\bar{r}}m^jl^{\bar{s}}m^k=\Xi _{\bar{r}j\bar{%
p}k}\overline{C_{sh}^p}m^{\bar{r}}m^jm^kl^{\bar{s}}=0$ \newline
\smallskip
we have,
\begin{eqnarray}
B_{|\bar{h}}|_{\bar{0}} &=&\frac 12B_{|\bar{h}}-\frac B2[-\frac{\bar{J}+\bar{%
Y}}2+(\bar{J}+\bar{Y})|_{\bar{0}}]l_{\bar{h}}  \label{2.9;;;} \\
&&-\frac B2[\frac{\bar{V}+\bar{H}}2+(\bar{V}+\bar{H})|_{\bar{0}}]m_{\bar{h}},
\nonumber
\end{eqnarray}
and its conjugate.

The conjugates of (\ref{2.9''}), (\ref{2.9;;;}) and Theorem 4.1 ii) allow us
to write
\begin{eqnarray}
\bar{B}_{|k}|_j &=&\frac 1{2F}\{\bar{B}_{|k}-\bar{B}[-\frac{J+Y}2+(J+Y)|_0%
]l_k  \label{2.9.} \\
&&-\bar{B}[\frac{V+H}2+(V+H)|_0]m_k\}l_j  \nonumber \\
&&-\{\frac 12[-\mathbf{I}(J+Y)+\bar{B}(J+Y)|_sm^s-\frac{\bar{B}}F(V+H)]l_k
\nonumber \\
&&+\frac 12[(-\mathbf{I+}\frac{B\bar{B}}2)(V+H)+\bar{B}(V+H)|_sm^s  \nonumber
\\
&&+2\frac AF\bar{B}_{|0}+2B\bar{B}_{|s}m^s+A\bar{B}(J+Y)]m_k+\mathbf{I}%
_{|k}+\frac B2\bar{B}_{|k}\}m_j.  \nonumber
\end{eqnarray}

It is also worthwhile to note the following identity
\begin{eqnarray}
\bar{B}|_{j|k} &=&\frac 1{2F}\bar{B}_{|k}l_j  \label{7} \\
&&-\{\mathbf{I}_{|k}+\frac{\bar{B}}2B_{|k}+\frac B2\bar{B}_{|k}+\frac 12(-%
\mathbf{I-}\frac{B\bar{B}}2)[(J+Y)l_k+(V+H)m_k]\}m_j,  \nonumber
\end{eqnarray}
which is obtained from (\ref{2.4''}), (\ref{2.4'''}) and (\ref{2.5'}).

Therefore, (\ref{2.9.}) and (\ref{7}) lead to
\begin{eqnarray}
\bar{B}|_{j|k}-\bar{B}_{|k}|_j &=&C_{jk}^i\bar{B}_{|i}+\frac{\bar{B}}%
2\{\frac 1F[-\frac{J+Y}2+(J+Y)|_0]l_kl_j  \label{7'} \\
&&+\frac 1F[\frac{V+H}2+(V+H)|_0]m_kl_j  \nonumber \\
&&+[\frac B2(J+Y)+(J+Y)|_sm^s-\frac 1F(V+H)]l_km_j  \nonumber \\
&&+[B(V+H)+(V+H)|_sm^s+A(J+Y)]m_km_j-B_{|k}m_j\},  \nonumber
\end{eqnarray}
because $C_{jk}^i\bar{B}_{|i}=(\frac AF\bar{B}_{|0}+B\bar{B}_{|s}m^s)m_km_j.$

\subsection{The $h\bar{v}-$ Riemann type tensor}

Now let us consider the $h\bar{v}-$ Riemann type tensor $P_{\bar{r}j%
\overline{h}k}.$ By Proposition 2.1.ii) and formulas (\ref{2.4}) and
(\ref{2.4''}), it results that
\begin{equation}
P_{\bar{r}0\overline{h}k}=-F[\bar{A}_{|k}+\bar{A}(J+Y)l_k+\bar{A}(V+H)m_k]m_{%
\bar{r}}m_{\bar{h}}.  \label{2.8''}
\end{equation}

But, Proposition 2.1 vi) allows us to reconstruct $P_{\bar{r}j\overline{h}%
k}. $ Indeed,
\begin{equation}
P_{\bar{r}j\overline{h}k}=P_{\bar{r}0\overline{h}k}|_j+P_{\bar{r}0\overline{h%
}s}C_{kj}^s  \label{2.9}
\end{equation}
and from (\ref{2.8''}), we obtain
\begin{eqnarray}
P_{\bar{r}0\overline{h}s}C_{kj}^s &=&-F[\frac AF\bar{A}_{|0}+B\bar{A}_{|s}m^s
\label{2.9'} \\
&&+A\bar{A}(J+Y)+B\bar{A}(V+H)]m_{\bar{r}}m_{\bar{h}}m_km_j  \nonumber
\end{eqnarray}
and
\begin{eqnarray}
P_{\bar{r}0\overline{h}k}|_j &=&-\{-\frac 12\bar{A}_{|k}l_j+FB\bar{A}%
_{|k}m_j+F\bar{A}_{|k}|_j-\bar{A}(J+Y)l_kl_j  \label{2.99} \\
&&+\bar{A}[FB(J+Y)-(V+H)]l_km_j+\frac{F\bar{A}B}2(V+H)m_km_j  \nonumber \\
&&+F[\bar{A}|_j(J+Y)+\bar{A}(J+Y)|_j]l_k  \nonumber \\
&&+F[\bar{A}|_j(V+H)+\bar{A}(V+H)|_j]m_k\}m_{\bar{r}}m_{\bar{h}}.  \nonumber
\end{eqnarray}

Plugging (\ref{2.9'}) and (\ref{2.99}) into (\ref{2.9}), gives
\begin{eqnarray}
P_{\bar{r}j\overline{h}k} &=&-\{-\frac 12\bar{A}_{|k}l_j+FB\bar{A}_{|k}m_j+F%
\bar{A}_{|k}|_j-\bar{A}(J+Y)l_kl_j  \label{2.9'''} \\
&&+\bar{A}[FB(J+Y)-(V+H)]l_km_j  \nonumber \\
&&+[A\bar{A}_{|0}+FB\bar{A}_{|s}m^s+FA\bar{A}(J+Y)+\frac{3F\bar{A}B}%
2(V+H)]m_km_j  \nonumber \\
&&+F[\bar{A}|_j(J+Y)+\bar{A}(J+Y)|_j]l_k  \nonumber \\
&&+F[\bar{A}|_j(V+H)+\bar{A}(V+H)|_j]m_k\}m_{\bar{r}}m_{\bar{h}}.  \nonumber
\end{eqnarray}
Recall the following property, $P_{\overline{r}j\overline{h}k}=\Xi _{j%
\overline{r}k\overline{h}}=\overline{\Xi _{\overline{j}r\overline{k}h}}$.
Writing it by means of (\ref{2.8'''}) and (\ref{2.9'''}), we obtain the
conditions
\begin{eqnarray}
\overline{A}_{|k}|_0 &=&\frac 32\overline{A}_{|k}\;;  \label{2.10''} \\
\overline{A}_{|k}|_jm^j &=&\frac 1F\overline{B}_{|k}-B\overline{A}_{|k}-[%
\frac{\bar{B}}{2F}(J+Y)+\bar{A}(J+Y)|_sm^s-\frac{\bar{A}}F(V+H)]l_k
\nonumber \\
&&-[(\frac{\bar{B}}{2F}+\frac{\bar{A}B}2)(V+H)+\bar{A}(V+H)|_sm^s+A\bar{A}%
(J+Y)  \nonumber \\
&&+\frac AF\overline{A}_{|0}+B\overline{A}_{|s}m^s]m_k  \nonumber
\end{eqnarray}
and theirs conjugates.

From both formulas (\ref{2.10''}), it follows that
\begin{eqnarray}
\overline{A}_{|k}|_j &=&\frac 3{2F}\overline{A}_{|k}l_j+\{\frac 1F\overline{B%
}_{|k}-B\overline{A}_{|k}  \label{0} \\
&&-\mathbf{[}\frac{\bar{B}}{2F}(J+Y)+\bar{A}(J+Y)|_sm^s-\frac{\bar{A}}F(V+H)%
\mathbf{]}l_k  \nonumber \\
&&-[(\frac{\bar{B}}{2F}+\frac{\bar{A}B}2)(V+H)+\bar{A}(V+H)|_sm^s+A\bar{A}%
(J+Y)  \nonumber \\
&&+\frac AF\overline{A}_{|0}+B\overline{A}_{|s}m^s]m_k\}m_j.  \nonumber
\end{eqnarray}

Moreover, from (\ref{2.4''}), (\ref{2.4'''}) and (\ref{2.5'}), we have
\begin{eqnarray}
\overline{A}|_{j|k} &=&\frac 3{2F}\overline{A}_{|k}l_j+\{-B\overline{A}_{|k}-%
\overline{A}B_{|k}+\frac 1F\overline{B}_{|k}  \label{1} \\
&&-(\frac{\bar{B}}{2F}-\frac{\bar{A}B}2)[(J+Y)l_k+(V+H)m_k\mathbf{]\}}m_j.
\nonumber
\end{eqnarray}

By subtracting (\ref{0}) from (\ref{1}), we get
\begin{eqnarray}
\overline{A}|_{j|k}-\overline{A}_{|k}|_j &=&C_{jk}^i\overline{A}_{|i}
\label{1'} \\
&&-\overline{A}\{B_{|k}-\mathbf{[}\frac B2(J+Y)+(J+Y)|_sm^s-\frac 1F(V+H)%
\mathbf{]}l_k  \nonumber \\
&&-[B(V+H)+(V+H)|_sm^s+A(J+Y)]m_k\}m_j,  \nonumber
\end{eqnarray}
because $C_{jk}^i\overline{A}_{|i}=(\frac AF\overline{A}_{|0}+B\overline{A}%
_{|s}m^s)m_km_j.$

\begin{theorem}
Let $(M,F)$ be a 2 - dimensional complex Finsler space. Then it is

i) purely Hermitian, or

ii) with $|A|\neq 0,$ $B=0$ and
\begin{equation}
(J+Y)|_sm^s=\frac 1F(V+H)\;;\;\;(V+H)|_sm^s=-A(J+Y),  \label{2}
\end{equation}

or

iii) with $AB^2\neq 0$ and
\begin{eqnarray}
(J+Y)|_0 &=&\frac{J+Y}2\;;\;\;(V+H)|_0=-\frac{V+H}2\;;  \label{3} \\
B_{|k} &=&[\frac B2(J+Y)+(J+Y)|_sm^s-\frac 1F(V+H)]l_k  \nonumber \\
&&+[B(V+H)+(V+H)|_sm^s+A(J+Y)]m_k.  \nonumber
\end{eqnarray}
\end{theorem}

\begin{proof}
Writing the identity i) from Proposition 2.2 for the vertical terms $\overline{A}$
and $\overline{B}$ it involves $\overline{A}|_{k|j}-\overline{A}%
_{|j}|_k=C_{jk}^i\overline{A}_{|i}$ and $\overline{B}|_{k|j}-\overline{B}%
_{|j}|_k=C_{jk}^i\overline{B}_{|i}$. But, taking into account (\ref{1'}) and
(\ref{7'}), it follows
\begin{eqnarray}
&&\overline{A}\{B_{|k}-\mathbf{[}\frac B2(J+Y)+(J+Y)|_sm^s-\frac 1F(V+H)%
\mathbf{]}l_k  \label{4} \\
&&-[B(V+H)+(V+H)|_sm^s+A(J+Y)]m_k\}m_j=0  \nonumber
\end{eqnarray}
and
\begin{eqnarray}
&&\bar{B}\{\frac 1F[-\frac{J+Y}2+(J+Y)|_0]l_kl_j  \label{5} \\
&&+\frac 1F[\frac{V+H}2+(V+H)|_0]m_kl_j  \nonumber \\
&&+[\frac B2(J+Y)+(J+Y)|_sm^s-\frac 1F(V+H)]l_km_j  \nonumber \\
&&+[B(V+H)+(V+H)|_sm^s+A(J+Y)]m_km_j-B_{|k}m_j\}=0.  \nonumber
\end{eqnarray}

Hence, we have the cases:

1. If $\overline{A}=0$ then by means of Proposition 4.1 i) gives the
statement i), or

2. If  $\overline{A}\neq 0$ then $|A|\neq 0$ and by (\ref{4}) we obtain
\begin{eqnarray}
B_{|k} &=&[\frac B2(J+Y)+(J+Y)|_sm^s-\frac 1F(V+H)]l_k  \label{6} \\
&&+[B(V+H)+(V+H)|_sm^s+A(J+Y)]m_k,  \nonumber
\end{eqnarray}
which substituted into (\ref{5}) leads to
\begin{equation}
\bar{B}\{[-\frac{J+Y}2+(J+Y)|_0]l_kl_j+[\frac{V+H}2+(V+H)|_0]m_kl_j\}=0.
\nonumber
\end{equation}
From here, it results either $\bar{B}=0$ which together with (\ref{6}) gives ii) or $\bar{B}\neq 0$. In this last case we have $|B|\neq 0$, and by (\ref{6}) results $%
(J+Y)|_0=\frac{J+Y}2\;$and$\;(V+H)|_0=-\frac{V+H}2,$ which with (\ref{5}) imply iii).

The independence of the above statement to the changes of local charts results by straightforward computations using (\ref{c}).
\end{proof}

\subsection{Two dimensional complex Berwald and Landsberg spaces}

The above considerations offer us the premises for some special
characterizations of the 2 - dimensional complex Berwald and Landsberg
spaces. Firstly, we write the identity iv) of Proposition 2.1 in terms of
the local complex Berwald frame. Some computations give

$\dot{\partial}_{\bar{h}}L_{jk}^i=\{[\bar{l}(J)+\frac 1{2F}J]l^il_jl_k+[\bar{%
l}(U)-\frac 1{2F}U]l^im_jl_k+[\bar{l}(V)-\frac 1{2F}V]l^il_jm_k$

$+[\bar{l}(X)-\frac 1{2F}X]l^im_jm_k+[\bar{l}(O)+\frac 3{2F}O]m^il_jl_k+[%
\bar{l}(Y)+\frac 1{2F}Y]m^im_jl_k$

$+[\bar{l}(E)+\frac 1{2F}E]m^il_jm_k+[\bar{l}(H)-\frac 1{2F}H]m^im_jm_k\}l_{%
\bar{h}}$

$+\{[\bar{m}(J)-\frac 1FO]l^il_jl_k+[\bar{m}(U)-\frac 1F(Y-J)+\frac 12\bar{B}%
U]l^im_jl_k$

$+[\bar{m}(V)-\frac 1F(E-J)+\frac 12\bar{B}V]l^il_jm_k+[\bar{m}(X)-\frac
1F(H-U-V)+\bar{B}X]l^im_jm_k$

$+[\bar{m}(O)-\frac 12\bar{B}O]m^il_jl_k+[\bar{m}(Y)+\frac 1FO]m^im_jl_k+[%
\bar{m}(E)+\frac 1FO]m^il_jm_k$

$+[\bar{m}(H)+\frac 1F(Y+E)+\frac 12\bar{B}H]m^im_jm_k\}m_{\bar{h}}.$

Using $\Xi _{\overline{r}j\overline{h}k}=-C_{j\bar{r}k|\bar{h}}$ and (\ref
{2.8'''}) it results

$C_{j\bar{r}k|\bar{h}}=[A_{|\bar{h}}l_{\bar{r}}+A(\bar{J}+\bar{Y})l_{\bar{r}%
}l_{\bar{h}}+A(\bar{V}+\bar{H})l_{\bar{r}}m_{\bar{h}}+B_{|\bar{h}}m_{\bar{r}%
}+\frac B2(\bar{J}+\bar{Y})m_{\bar{r}}l_{\bar{h}}+\frac B2(\bar{V}+\bar{H}%
)m_{\bar{r}}m_{\bar{h}}]m_jm_k.$

The above outcomes substituted into Proposition 2.1 iv), lead to

\begin{proposition}
Let $(M,F)$ be a $2$ - dimensional complex Finsler space. Then

i) $J|_{\bar{k}}=-\frac 1{2F}Jl_{\bar{k}}+\frac 1FOm_{\bar{k}};$ $V|_{\bar{k}%
}=\frac 1{2F}Vl_{\bar{k}}+[\frac 1F(E-J)-\frac 12\bar{B}V]m_{\bar{k}};$

i\medskip i) $\bar{l}(U)-\frac 1{2F}U=\bar{l}(X)-\frac 1{2F}X=\bar{l}%
(O)+\frac 3{2F}O=\bar{l}(Y)+\frac 1{2F}Y=\bar{l}(E)+\frac 1{2F}E$

$=\bar{l}(H)-\frac 1{2F}H=0;$

iii) $\bar{m}(U)-\frac 1F(Y-J)+\frac 12\bar{B}U+FA[\bar{m}(O)-\frac 12\bar{B}%
O]=0;$

iv) $\bar{m}(V)-\frac 1F(E-J)+\frac 12\bar{B}V=0;$

v) $\bar{m}(X)-\frac 1F(H-U-V)+\bar{B}X+FA[\bar{m}(E)+\frac 1FO]=0;$

vi) $\frac 1F\bar{A}_{|0}+\bar{A}(J+Y)=\bar{m}(O)-\frac 12\bar{B}O;$

vii) $\frac 1F\bar{B}_{|0}+\frac{\bar{B}}2(J+Y)=\bar{m}(Y)+\frac 1FO+FB[\bar{%
m}(O)-\frac 12\bar{B}O];$

viii) $\bar{A}_{|k}m^k+\bar{A}(V+H)=\bar{m}(E)+\frac 1FO;$

ix) $\bar{B}_{|k}m^k+\frac{\bar{B}}2(V+H)=\bar{m}(H)+\frac 1F(Y+E)+\frac 12%
\bar{B}H+FB[\bar{m}(E)+\frac 1FO].$
\end{proposition}

Next, we rewrite the identity v) from Proposition 2.1, $\dot{\partial}%
_hL_{jk}^i=C_{j\bar{r}h|k}g^{\bar{r}i}$ with respect to the complex Berwald
frame. Taking into account Proposition 3.1, we have

$\dot{\partial}_hL_{jk}^i=\{[l(U)+\frac 1{2F}U]l^im_jlm_k+[l(X)+\frac
3{2F}X]l^im_jm_k$

$+[l(Y)-\frac 1{2F}Y]m^im_jl_k+[l(H)+\frac 1{2F}H]m^im_jm_k\}l_h$

$+\{[m(U)-A(Y-J)+\frac 12BU-\frac 1FX]l^im_jl_k+[m(Y)+AO-\frac
1F(H-U)]m^im_jl_k$

$+[m(X)+A(U+V-H)+BX]l^im_jm_k$

$+[m(H)+A(Y+E)+\frac 1FX+\frac 12BH]m^im_jm_k\}m_h.$

On the other hand, $C_{j\bar{r}h|k}g^{\bar{r}i}=%
\{[A_{|k}-A(J+Y)l_k-A(V+H)m_k]l^i$

$+[B_{|k}-\frac B2(J+Y)l_k-\frac B2(V+H)m_k]m^i\}m_jm_h.$ From here we obtain

\begin{proposition}
Let $(M,F)$ be a $2$ - dimensional complex Finsler space. Then

i) $l(U)+\frac 1{2F}U=l(X)+\frac 3{2F}X=l(Y)-\frac 1{2F}Y=l(H)+\frac
1{2F}H=0;$

ii) $m(U)-A(Y-J)+\frac 12BU-\frac 1FX=\frac 1FA_{|0}-A(J+Y);$

iii) $m(Y)+AO-\frac 1F(H-U)=\frac 1FB_{|0}-\frac B2(J+Y);$

iv) $m(X)+A(U+V-H)+BX=A_{|k}m^k-A(V+H);$

v) $m(H)+A(Y+E)+\frac 1FX+\frac 12BH=B_{|k}m^k-\frac B2(V+H).$
\end{proposition}

We note that the assertions of Propositions 4.2 and 4.3 are preserved to
changes of local charts.

Now, taking into account (\ref{2.10}), we have $G^i=\frac{F^2}2(Jl^i+Om^i).$
By (\ref{2.4'}) and by Proposition 4.2 i), ii) and vi) result $\dot{\partial}%
_{\bar{h}}G^i=\frac{F^2}2[\bar{m}(O)-\frac 12\bar{B}O]m^im_{\bar{h}}=\frac
F2[\bar{A}_{|0}+F\bar{A}(J+Y)]m^im_{\bar{h}}.$ So, we have proved

\begin{lemma}
For any $2$ - dimensional complex Finsler space, $\dot{\partial}_{\bar{h}%
}G^i=0$ if and only if $\bar{m}(O)=\frac 12\bar{B}O,$ equivalently with $%
\bar{A}_{|0}+F\bar{A}(J+Y)=0.$
\end{lemma}

\begin{theorem}
If $(M,F)$ is a K\"{a}hler $2$ - dimensional complex Finsler space, then $%
\dot{\partial}_{\bar{h}}G^i=0.$
\end{theorem}

\begin{proof}
By Propositions 3.2, 4.2 iii) and iv) result $FA[\bar{m}%
(O)-\frac 12\bar{B}O]=0.$ So, we have either $A=0$ or $\bar{m}(O)=\frac 12%
\bar{B}O.$ If $A=0,$ by Proposition 4.2 vi) we obtain $\bar{m}(O)=\frac 12\bar{B}O$, which is globally. So
that $\dot{\partial}_{\bar{h}}G^i=0.$ If $\bar{m}(O)=\frac 12\bar{B}O,$ by Lemma 4.1, results $\dot{\partial}_{\bar{%
h}}G^i=0.$
\end{proof}

\begin{remark}
The above theorem shows that in dimension two, the class of the complex
Berwald spaces coincides with the class of K\"{a}hler spaces.
\end{remark}

\begin{theorem}
A $2$ - dimensional complex Finsler space is Berwald if and only if it is
weakly K\"{a}hler and $\dot{\partial}_{\bar{h}}G^i=0$.
\end{theorem}

\begin{proof}
The necessity is obvious. For sufficiency, using
Propositions 3.2. ii), 4.3 iii) and iv) and Lemma 4.1, it results $\bar{m}%
(V)-\frac 1F(Y-J)+\frac 12\bar{B}V=0$ and $\bar{m}(V)-\frac 1F(E-J)+\frac 12%
\bar{B}V=0.$ From here we obtain $Y=E,$ i.e. the space is K\"{a}hler, and
therefore Berwald.
\end{proof}

\begin{proposition}
If $(M,F)$ is a $2$ - dimensional complex Berwald space, then
\begin{eqnarray}
U|_{\bar{k}} &=&\frac 1{2F}Ul_{\bar{k}}+[\frac 1F(Y-J)-\frac 12\bar{B}U]m_{%
\bar{k}};\;Y|_{\bar{k}}=-\frac 1{2F}Yl_{\bar{k}}-\frac 1FOm_{\bar{k}};\;
\nonumber \\
O|_{\bar{k}} &=&-\frac 3{2F}Ol_{\bar{k}}+\frac 12\bar{B}Om_{\bar{k}};\;X|_{%
\bar{k}}=\frac 1{2F}Xl_{\bar{k}}+[\frac 1F(H-2V)-\bar{B}X]m_{\bar{k}};\;
\nonumber \\
H|_{\bar{k}} &=&\frac 1{2F}Hl_{\bar{k}}-(\frac 2FY+\frac 12\bar{B}H)m_{\bar{k%
}};  \label{IV.1}
\end{eqnarray}
equivalently with
\begin{eqnarray}
A_{|\bar{k}} &=&-A(\bar{J}+\bar{Y})l_{\bar{k}}-A(\bar{V}+\bar{H})m_{\bar{k}};
\label{IV.2} \\
B_{|\bar{k}} &=&-\frac B2(\bar{J}+\bar{Y})l_{\bar{k}}-\frac B2(\bar{V}+\bar{H%
})m_{\bar{k}};  \nonumber
\end{eqnarray}
equivalently with
\begin{eqnarray}
U|_k &=&-\frac 1{2F}Ul_k+[A(Y-J)-\frac 12BU+\frac 1FX]m_k;\;  \label{IV.3} \\
Y|_k &=&\frac 1{2F}Yl_k+[\frac 1F(H-U)-AO]m_k;\;  \nonumber \\
X|_k &=&-\frac 3{2F}Xl_k-[2AU-AH+BX]m_k;  \nonumber \\
H|_k &=&-\frac 1{2F}Hl_k-[2AY+\frac 12BH+\frac 1FX]m_k;  \nonumber
\end{eqnarray}
equivalently with
\begin{eqnarray}
A_{|k} &=&A(J+Y)l_k+A(V+H)m_k;  \label{IV.4} \\
B_{|k} &=&\frac B2(J+Y)l_k+\frac B2(V+H)m_k.  \nonumber
\end{eqnarray}
\end{proposition}

\begin{proof}
Under the assumption of Berwald, we have $\dot{\partial}_{\bar{h}%
}G^i=\dot{\partial}_{\bar{h}}N_k^i=\dot{\partial}_{\bar{h}}L_{jk}^i=0$ which
together with Proposition 4.2 induces (\ref{IV.1}). Using Theorem 2.2 and
Propositions 4.2 and 4.3 it results the equivalence between (\ref{IV.1}), (%
\ref{IV.2}), (\ref{IV.3}) and (\ref{IV.4}). By straightforward computations it results their global validity.
\end{proof}

We note that the equivalent sets of relations (\ref{IV.1}), (\ref{IV.2}), (\ref{IV.3}%
) and (\ref{IV.4}) have a geometric character and are only necessary
conditions for complex Berwald space. These become sufficient together
with weakly K\"{a}hler condition.

Trivial examples of complex Berwald spaces are given by the purely Hermitian
and locally Minkowski manifolds. An nontrivial example of $2$ - dimensional
complex Berwald space is welcomed.

Let $\Delta =\left\{ (z,w)\in \mathbf{C}^2,\;|w|<|z|<1\right\} $ be the
Hartogs triangle with the K\"{a}hler-purely Hermitian metric
\begin{equation}
a_{i\overline{j}}=\frac{\partial ^2}{\partial z^i\partial \overline{z}^j}%
(\log \frac 1{\left( 1-|z|^2\right) \left( |z|^2-|w|^2\right) });\mbox{
}\alpha ^2(z,w;\eta ,\theta )=a_{i\overline{j}}\eta ^i\overline{\eta }^j,\;
\label{III.10}
\end{equation}
where $z,$ $w,\eta ,$ $\theta $ are the local coordinates $z^1,$ $z^2,$ $%
\eta ^1,$ $\eta ^2,$ respectively, and $|z^i|^2:=z^i\bar{z}^i,$ $z^i\in
\{z,w\},$ $\eta ^i\in \{\eta ,\theta \}.$ We choose
\begin{equation}
b_z=\frac w{|z|^2-|w|^2};\;b_w=-\frac z{|z|^2-|w|^2}.  \label{III.11}
\end{equation}
With these tools we construct $\alpha (z,w,\eta ,\theta ):=\sqrt{a_{i\bar{j}%
}(z,w)\eta ^i\bar{\eta}^j}$ and $\beta (z,\eta )=b_i(z,w)\eta ^i$ and from
here we obtain the complex Randers metric $F=\alpha +|\beta |$ and the
complex Kropina metric $F:=\frac{\alpha ^2}{|\beta |}.$ By a direct
computation, we deduce
\begin{eqnarray}
a_{z\overline{z}} &=&\frac 1{\left( 1-|z|^2\right) ^2}+b_zb_{\bar{z}};\;a_{z%
\overline{w}}=b_zb_{\bar{w}};\;a_{w\overline{w}}=b_wb_{\bar{w}};
\label{III.12} \\
a^{\overline{z}z} &=&\left( 1-|z|^2\right) ^2;\;a^{\overline{w}z}=\frac{%
\overline{w}z\left( 1-|z|^2\right) ^2}{|z|^2};  \nonumber \\
a^{\overline{w}w} &=&\frac{\left( |z|^2-|w|^2\right) ^2}{|z|^2}+\frac{%
|w|^2\left( 1-|z|^2\right) ^2}{|z|^2};  \nonumber \\
b^z &=&0;\;b^w=-\frac{|z|^2-|w|^2}z;\;||b||^2=1;\;\alpha ^2-|\beta |^2=\frac{%
|\eta |^2}{\left( 1-|z|^2\right) ^2}  \nonumber
\end{eqnarray}
and the horizontal coefficients of the $C-F$ connection are
\begin{eqnarray}
L_{zz}^z &=&\frac{2\overline{z}}{1-|z|^2}\;;\;L_{zw}^z=L_{wz}^z=0\;;%
\;L_{zz}^w=\frac 1{1-|z|^2}+\frac 1{|z|^2-|w|^2}\;;\;  \nonumber \\
L_{zw}^w &=&L_{wz}^w=-\frac{|z|^2+|w|^2}{z\left( |z|^2-|w|^2\right) }%
\;;\;L_{ww}^w=\frac{2\overline{w}}{|z|^2-|w|^2}.  \nonumber
\end{eqnarray}
which attest the K\"{a}hler property. The spray coefficients
\[
\,G^z=\frac{\overline{z}\eta ^2}{1-|z|^2}\;;\;G^w=\frac{\overline{z}w}%
z\left( \frac 1{1-|z|^2}+\frac 1{|z|^2-|w|^2}\right) \eta ^2+\frac{\overline{%
w}}{|z|^2-|w|^2}\theta ^2
\]
are holomorphic in $\eta .$ Moreover, $K_{F,m}^v(z,\eta )=-\frac 2{F^2}<0.$

\begin{proposition}
If $(M,F)$ is a $2$ - dimensional complex Berwald space, then $\mathbf{I}%
_{|j}=0.$
\end{proposition}

\begin{proof}
Firstly, because the space is Berwald, the identity ii) from
Proposition 2.2 is $B|_{\bar{k}|j}-B_{|j}|_{\bar{k}}=0.$ On the other hand,
the relations (\ref{2.4''}), (\ref{2.4'''}), (\ref{2.5''}), (\ref{2.5'}) and
(\ref{IV.4}) lead to $B|_{\bar{k}|j}-B_{|j}|_{\bar{k}}=-\mathbf{I}_{|j}m_{%
\bar{k}}.$ So, $\mathbf{I}_{|j}=0.$
\end{proof}

The converse of the above Proposition is not true. There exist $2$ -
dimensional complex Finsler spaces with $\mathbf{I}_{|j}=0$ which are not
Berwald. We attest this fact by an example. Namely, we consider the complex
version of \textit{Antonelli - Shimada }metric
\begin{equation}
F_{AS}^2=L_{AS}(z,w;\eta ,\theta ):=e^{2\sigma }\left( |\eta |^4+|\theta
|^4\right) ^{\frac 12},\;\mbox{with}\;\;\eta ,\theta \neq 0,\;\;
\label{IV.9}
\end{equation}
on a domain $D$ from $\widetilde{T^{\prime }M},$ $\dim _CM=2,$ such that its
metric tensor is nondegenerated. We relabeled the local coordinates $z^1,$ $%
z^2,$ $\eta ^1,$ $\eta ^2$ as $z,$ $w,\eta ,$ $\theta ,$ respectively. $%
\sigma (z,w)$ is a real valued function and $|\eta ^i|^2:=\eta ^i\overline{%
\eta }^i,$ $\eta ^i\in \{\eta ,\theta \}$, (\cite{Mub}).

A direct computation leads to
\begin{eqnarray*}
g_{z\overline{z}} &:&=g_{1\overline{1}}=\frac{e^{8\sigma }|\eta |^2(|\eta
|^4+2|\theta |^4)}{L_{AS}^3};\;g^{\overline{z}z}:=g^{\overline{1}1}=\frac{%
2|\eta |^4+|\theta |^4}{2|\eta |^2L_{AS}}; \\
g_{z\overline{w}} &:&=g_{1\overline{2}}=-\frac{e^{8\sigma }|\eta |^2|\theta
|^2\overline{\eta }\theta }{L_{AS}^3};\;g^{\overline{w}z}:=g^{\overline{2}1}=%
\frac{\eta \overline{\theta }}{2L_{AS}}; \\
g_{w\overline{w}} &:&=g_{2\overline{2}}=\frac{e^{8\sigma }|\theta |^2(2|\eta
|^4+|\theta |^4)}{L_{AS}^3};\;g^{\overline{w}w}:=g^{\overline{2}2}=\frac{%
|\eta |^4+2|\theta |^4}{2|\theta |^2L_{AS}}; \\
\Delta ^2 &=&\det \left( g_{i\overline{j}}\right) =\frac{2e^{8\sigma }|\eta
|^2|\theta |^2}{L_{AS}^2}; \\
l^z &:&=l^1=\frac \eta {F_{AS}};\;l^w:=l^2=\frac \theta {F_{AS}}; \\
l_z &:&=l_1=\frac{e^{4\sigma }|\eta |^2\overline{\eta }}{F_{AS}^3}%
;\;l_w:=l_2=\frac{e^{4\sigma }|\theta |^2\overline{\theta }}{F_{AS}^3};\; \\
m^z &:&=m^1=-\frac{|\theta |\overline{\theta }}{\sqrt{2}|\eta |F_{AS}}%
;\;m^w:=m^2=\frac{|\eta |\overline{\eta }}{\sqrt{2}|\theta |F_{AS}}; \\
m_z &:&=m_1=-\frac{2e^{4\sigma }|\eta ||\theta |\theta }{\sqrt{2}F_{AS}^3}%
;\;m_w:=m_2=\frac{2e^{4\sigma }|\eta ||\theta |\eta }{\sqrt{2}F_{AS}^3}.
\end{eqnarray*}

The nonzero coefficients of the $C-F$ connection are
\begin{eqnarray}
L_{zz}^z &=&L_{wz}^w=2\frac{\partial \sigma }{\partial z}%
;\;L_{zw}^z=L_{ww}^w=2\frac{\partial \sigma }{\partial w};  \label{IV.11} \\
C_{zz}^z &=&\frac{e^{8\sigma }|\theta |^8\overline{\eta}}{|\eta |^2L_{AS}^4}%
;\;C_{zw}^z=C_{wz}^z=-\frac{e^{8\sigma }|\theta |^6\overline{\theta}}{%
L_{AS}^4};\;C_{ww}^z=\frac{e^{8\sigma }|\theta |^4\overline{\theta}^2\eta }{%
L_{AS}^4};  \nonumber \\
C_{ww}^w &=&\frac{e^{8\sigma }|\eta |^8\overline{\theta}}{|\theta |^2L_{AS}^4%
};\;C_{zw}^w=C_{wz}^w=-\frac{e^{8\sigma }|\eta |^6\overline{\eta}}{L_{AS}^4}%
;\;C_{zz}^w=\frac{e^{8\sigma }|\eta |^4\overline{\eta}^2\theta }{L_{AS}^4}.
\nonumber
\end{eqnarray}

From here we obtain
\[
A=\frac{\overline{\eta }^2\overline{\theta }^2}{2|\eta |^2|\theta |^2F_{AS}}%
;\;B=\frac{\overline{\eta }\overline{\theta }(|\eta |^4-|\theta |^4)}{\sqrt{2%
}|\eta |^3|\theta |^3F_{AS}}\neq 0;\;\mathbf{I}=\frac 2{L_{AS}};\;
\]
and so $\mathbf{I}_{|k}=0.$ Moreover, $K_{F_{AS},m}^v(z,\eta )=\frac
4{L_{AS}}>0.$ The local coefficients $L_{jk}^i$ depend only on $z$ and $w$,
but the Antonelli - Shimada metric is not Berwald because, in generally, it
is not K\"{a}hler. If $\sigma $ is a constant, then the Antonelli - Shimada
metric is Berwald and locally Minkowski.

The above examples suggest us to pay attention to the class of $2$ -
dimensional complex Finsler spaces with $\mathbf{I}|_i=-\frac{\mathbf{I}}%
Fl_i $ and $\mathbf{I}_{|k}=0.$ With these assumptions we obtain $\frac{%
\partial \mathbf{I}}{\partial z^k}=N_k^r\mathbf{I}|_r=-\frac{\mathbf{I}}%
Fl_rN_k^r=-\frac{\mathbf{I}}L\frac{\partial L}{\partial z^k}.$ From here it
results $L\frac{\partial \mathbf{I}}{\partial z^k}+\mathbf{I}\frac{\partial L%
}{\partial z^k}=0$ and so, $\frac{\partial (\mathbf{I}L)}{\partial z^k}=0.$
Thus, we have proved

\begin{theorem}
Let $(M,F)$ be a connected 2 - dimensional complex Finsler space with $%
\mathbf{I}_{|k}=0,$ $\mathbf{I}|_i=-\frac{\mathbf{I}}Fl_i$ and $AB^2\neq 0.$
Then $\mathbf{I}L$ is a constant on $(M,F)$ and $K_{F,m}^v(z,\eta )=\frac{2c}%
L,$ where $c\in \mathbf{R}.$
\end{theorem}

\begin{proof}
Indeed, from the above considerations we have $\frac{\partial (\mathbf{I}L)}{%
\partial z^k}=0.$ Therefore, $\mathbf{I}L$ does not depend on $z.$ Hence $%
\mathbf{I}L=c(\eta ,\bar{\eta}),$ where $c(\eta ,\bar{\eta})$ is real
valued. Differentiating, we obtain $\mathbf{I}|_iL+F\mathbf{I}l_i=c|_i.$
But, $\mathbf{I}|_i=-\frac{\mathbf{I}}Fl_i.$ Hence $c|_i=0$ and its
conjugate, which means that $c$ is a constant. It results that $\mathbf{I}%
=\frac cL.$
\end{proof}

In order to investigate $2$ - dimensional complex Landsberg spaces, we
translate the $R\Gamma $ and $B\Gamma $ connections in terms of the local
complex Berwald frames. After some computations we obtain
\begin{eqnarray}
\stackrel{c}{L_{jk}^i} &=&Jl^il_jl_k+\frac{U+V}2(l^im_jl_k+l^il_jm_k)+[X-%
\frac{FA}2(Y-E)]l^im_jm_k  \label{IV.12} \\
&&+Om^il_jl_k+\frac{Y+E}2(m^im_jl_k+m^il_jm_k)+[H-\frac{FB}2(Y-E)]m^im_jm_k
\nonumber
\end{eqnarray}
and
\begin{eqnarray}
\stackrel{B}{L_{jk}^i} &=&Jl^il_jl_k+\frac{U+V}2(l^im_jl_k+l^il_jm_k)
\label{IV.13} \\
&&+\{X+\frac 12[A_{|0}-FA(J+Y)]\}l^im_jm_k+Om^il_jl_k  \nonumber \\
&&+\frac{Y+E}2(m^im_jl_k+m^il_jm_k)+\{H+\frac 12[B_{|0}-\frac{FB}%
2(J+Y)]\}m^im_jm_k.  \nonumber
\end{eqnarray}

By (\ref{IV.12}), (\ref{IV.13}) and (\ref{c}) immediately results

\begin{theorem}
A $2$ - dimensional complex Finsler space is Landsberg if and only if $%
FA(E-Y)=A_{|0}-FA(J+Y)\;$and$\;FB(E-Y)=B_{|0}-\frac{FB}2(J+Y).$
\end{theorem}

\begin{proposition}
If $(M,F)$ is a $2$ - dimensional complex Finsler space weakly K\"{a}hler
with $B=0$ then it is Landsberg.
\end{proposition}

\begin{proof}
Because $U=V$ and $B=0,$ by Proposition 4.1 ii) and
Proposition 4.3 ii) we have $FA(E-Y)=A_{|0}-FA(J+Y),$ i.e. the space is
Landsberg.
\end{proof}

\subsection{The $h\bar{h}-$ Riemann type tensor}

Let us investigate the $h\bar{h}-$ Riemann type tensor $R_{\bar{r}j\overline{%
h}k}.$ By (\ref{1.4'}), (\ref{2.4}) and (\ref{2.5}) we can write

$R_{\bar{r}j\overline{h}k}=g_{i\bar{r}}R_{j\bar{h}k}^i$

$=-\left( l_il_{\bar{r}}+m_im_{\bar{r}}\right) \{(l_{\bar{h}}\bar{\lambda}%
+m_{\bar{h}}\bar{\mu})(L_{jk}^i)$

$+[(l_{\bar{h}}\bar{\lambda}+m_{\bar{h}}\bar{\mu}%
)(N_k^n)](Al^im_jm_n+Bm^im_jm_n)\}$

$=-\left( l_il_{\bar{r}}+m_im_{\bar{r}}\right) l_{\bar{h}}[\bar{\lambda}%
(L_{jk}^i)+F\bar{\lambda}(L_{sk}^n)l^s(Al^i+Bm^i)m_jm_n]$

$-\left( l_il_{\bar{r}}+m_im_{\bar{r}}\right) m_{\bar{h}}[\bar{\mu}%
(L_{jk}^i)+F\bar{\mu}(L_{sk}^n)l^s(Al^i+Bm^i)m_jm_n].$ It results that
\begin{eqnarray}
R_{\bar{r}j\bar{h}k} &=&-[\bar{\lambda}(l_iL_{jk}^i)+FA\bar{\lambda}%
(L_{sk}^nl^s)m_jm_n]l_{\bar{r}}l_{\bar{h}}  \label{2.12} \\
&&-[\bar{\lambda}(L_{jk}^i)m_i+FB\bar{\lambda}(L_{sk}^nl^s)m_jm_n]m_{\bar{r}%
}l_{\bar{h}}  \nonumber \\
&&-[\bar{\mu}(l_iL_{jk}^i)+FA\bar{\mu}(L_{sk}^nl^s)m_jm_n]l_{\bar{r}}m_{\bar{%
h}}  \nonumber \\
&&-[\bar{\mu}(L_{jk}^i)m_i+FB\bar{\mu}(L_{skn}^nl^s)m_jm]m_{\bar{r}}m_{\bar{h%
}}.  \nonumber
\end{eqnarray}

Further on, our goal is to find the link between the horizontal covariant
derivatives of the functions (\ref{2.11}) and theirs properties. Indeed, from
(\ref{2.12}) it follows that $R_{\bar{0}0\bar{h}0}=-LF\bar{\lambda}%
(l^jl^kl_iL_{jk}^i)l_{\bar{h}}-LF\bar{\mu}(l^jl^kl_iL_{jk}^i)m_{\bar{h}%
}=-LJ_{|\bar{0}}l_{\bar{h}}-LFJ_{|\bar{s}}m^{\bar{s}}m_{\bar{h}}$ and $R_{%
\bar{0}0\bar{0}k}=-LF\bar{\lambda}(l^jl_iL_{jk}^i).$ The property $\overline{%
R_{\bar{0}0\bar{k}0}}=R_{\bar{0}0\bar{0}k}$ leads to $F\bar{\lambda}%
(l^jl_iL_{jk}^i)=\bar{J}_{|0}l_k+F\bar{J}_{|s}m^sm_k$, which gives
\begin{equation}
\bar{J}_{|0}=J_{|\bar{0}}\;;\;\;\;\bar{J}_{|s}m^s=\frac 1FV_{|\bar{0}}+\frac
12V(\bar{J}+\bar{Y}).  \label{2.13''}
\end{equation}

Moreover, by (\ref{2.12})

$R_{\bar{r}0\bar{0}0}=-LF\bar{\lambda}(l^jl^kl_iL_{jk}^i)l_{\bar{r}}-LF\bar{%
\lambda}(l^jl^km_iL_{jk}^i)m_{\bar{r}}+\frac 12LFO(\bar{J}+\bar{Y})m_{\bar{r}%
}$

$=-LJ_{|\bar{0}}l_{\bar{r}}-LO_{|\bar{0}}m_{\bar{r}}+\frac 12LFO(\bar{J}+%
\bar{Y})m_{\bar{r}}$ and

$R_{\bar{0}r\bar{0}0}=-LF\bar{\lambda}(l^kl_iL_{rk}^i)-L^2A\bar{\lambda}%
(l^kl^sm_nL_{sk}^n)m_r+\frac 12L^2AO(\bar{J}+\bar{Y})m_r$

$=-LF\bar{\lambda}(l^kl_iL_{rk}^i)-LFAO_{|\bar{0}}m_r+\frac 12L^2AO(\bar{J}+%
\bar{Y})m_r.$

But, $\overline{R_{\bar{r}0\bar{0}0}}=R_{r\bar{0}0\bar{0}}=R_{\bar{0}r\bar{0}%
0}$ leads to

$\bar{J}_{|0}l_r+[\bar{O}_{|0}-\frac 12F\bar{O}(J+Y)]m_r=F\bar{\lambda}%
(l^kl_iL_{rk}^i)+FA[O_{|\bar{0}}-\frac 12FO(\bar{J}+\bar{Y})]m_r.$

The contraction with $m^r$ gives
\begin{equation}
\bar{O}_{|0}-\frac 12F\bar{O}(J+Y)-FAO_{|\bar{0}}+\frac 12LAO(\bar{J}+\bar{Y}%
)=U_{|\bar{0}}+\frac 12FU(\bar{J}+\bar{Y}).  \label{2.13;}
\end{equation}
Next, from (\ref{2.12}) we have

$R_{\bar{r}0\bar{h}0}m^{\bar{r}}=-L\bar{\lambda}(l^jl^kL_{jk}^i)m_il_{\bar{h}%
}-L\bar{\mu}(l^jl^kL_{jk}^i)m_im_{\bar{h}}$

$=-F[O_{|\bar{0}}l_{\bar{h}}-\frac 12FO(\bar{J}+\bar{Y})]l_{\bar{h}}-L[O_{|%
\bar{s}}m^{\bar{s}}-\frac 12O(\bar{V}+\bar{H})]m_{\bar{h}}.$

On the other hand

$R_{\bar{0}r\bar{0}h}m^r=-L\bar{\lambda}(m^rl_iL_{rh}^i)-\frac 12L(\bar{J}+%
\bar{Y})(Ul_h+Xm_h)-LFA\bar{\lambda}(l^sm_nL_{sh}^n)$

$+\frac 12LFA(\bar{J}+\bar{Y})(Ol_h+Em_h).$

Using $\overline{R_{\bar{r}0\bar{h}0}m^{\bar{r}}}=R_{r\bar{0}h\bar{0}}m^r=R_{%
\bar{0}r\bar{0}h}m^r$, we obtain

$[\bar{O}_{|0}-\frac 12F\bar{O}(J+Y)]l_h+F[\bar{O}_{|s}m^s-\frac 12\bar{O}%
(V+H)]m_h$

$=F\bar{\lambda}(m^rl_iL_{rh}^i)+LA\bar{\lambda}(l^sm_nL_{sh}^n)+\frac 12F(%
\bar{J}+\bar{Y})(Ul_h+Xm_h)$

$-\frac 12LA(\bar{J}+\bar{Y})(Ol_h+Em_h),$

which by transvection with $m^h$ gives
\begin{equation}
\bar{O}_{|s}m^s-\frac 12\bar{O}(V+H)-AE_{|\bar{0}}=\frac 1FX_{|\bar{0}}+X(%
\bar{J}+\bar{Y}).  \label{2.13;;}
\end{equation}
Taking again into account (\ref{2.12}), it follows

$R_{\bar{0}0\bar{h}k}m^k=-L\bar{\lambda}(l^jl_iL_{jk}^i)m^kl_{\bar{h}}-L\bar{%
\mu}(l^jl_iL_{jk}^i)m^km_{\bar{h}}$

$=-F[V_{|\bar{0}}+\frac 12FV(\bar{J}+\bar{Y})]l_{\bar{h}}-L[V_{|\bar{s}}m^{%
\bar{s}}+\frac 12V(\bar{V}+\bar{H})]m_{\bar{h}}$ and

$R_{\bar{0}0\bar{k}h}m^{\bar{k}}=-L\bar{\mu}(l^jl_iL_{jh}^i).$

These relations together with $\overline{R_{\bar{0}0\bar{h}k}m^km^{\bar{h}}}%
=R_{0\bar{0}h\bar{k}}m^{\bar{k}}m^h=R_{\bar{0}0\bar{k}h}m^{\bar{k}}m^h$ give
\begin{equation}
\bar{V}_{|s}m^s+\frac 12\bar{V}(V+H)=V_{|\bar{s}}m^{\bar{s}}+\frac 12V(\bar{V%
}+\bar{H}).  \label{2.13;;;}
\end{equation}
Next, (\ref{2.12}) involves

$R_{\bar{r}0\bar{h}k}m^{\bar{r}}m^k=-F\bar{\lambda}(l^jm^km_iL_{jk}^i)l_{%
\bar{h}}-F\bar{\mu}(l^jm^km_iL_{jk}^i)m_{\bar{h}}$

$=-E_{|\bar{0}}l_{\bar{h}}-FE_{|\bar{s}}m^{\bar{s}}m_{\bar{h}}$ and

$R_{\bar{0}r\bar{k}h}m^rm^{\bar{k}}=-F\bar{\mu}(l_iL_{rh}^i)m^r-LA\bar{\mu}%
(l^sL_{sh}^n)m_n.$

But, $\overline{R_{\bar{r}0\bar{h}k}m^{\bar{r}}m^k}=R_{r\bar{0}h\bar{k}%
}m^rm^{\bar{k}}=R_{\bar{0}r\bar{k}h}m^rm^{\bar{k}}$ so that

$-\bar{E}_{|0}l_h-F\bar{E}_{|s}m^sm_h=-F\bar{\mu}(l_iL_{rh}^i)m^r-LA\bar{\mu}%
(l^sL_{sh}^n)m_n.$

By transvection with $l^h$ and $m^h$ we obtain
\begin{eqnarray}
\frac 1F\bar{E}_{|0}-FAO_{|\bar{s}}m^{\bar{s}}+\frac 12FAO(\bar{V}+\bar{H})
&=&U_{|\bar{s}}m^{\bar{s}}+\frac 12U(\bar{V}+\bar{H});  \label{2.13;;;;} \\
\bar{E}_{|s}m^s-FAE_{|\bar{s}}m^{\bar{s}} &=&X_{|\bar{s}}m^{\bar{s}}+X(\bar{V%
}+\bar{H}).  \nonumber
\end{eqnarray}
Using again (\ref{2.12}), we have

$R_{\bar{r}j\bar{h}k}m^{\bar{r}}m^jm^k=-[\bar{\lambda}(L_{jk}^i)m_im^jm^k+FB%
\bar{\lambda}(l^sm^km_nL_{sk}^n)]l_{\bar{h}}$

$-[\bar{\mu}(L_{jk}^i)m_im^jm^k+FB\bar{\mu}(l^sm^km_nL_{sk}^n)]m_{\bar{h}}$

$=-(\frac 1FH_{|\bar{0}}+\frac 12H(\bar{J}+\bar{Y})+BE_{|\bar{0}})l_{\bar{h}%
}-(H_{|\bar{s}}m^{\bar{s}}+\frac 12H(\bar{V}+\bar{H})+FBE_{|\bar{s}}m^{\bar{s%
}})m_{\bar{h}}.$ \newline
On the other hand,

$R_{\bar{j}r\bar{k}h}m^{\bar{j}}m^rm^{\bar{k}}=-\bar{\mu}(m^rm_iL_{rh}^i)-FB%
\bar{\mu}(l^sL_{sh}^n)m_n.$\newline
But, $\overline{R_{\bar{r}j\bar{h}k}m^{\bar{r}}m^jm^k}=R_{r\bar{j}h\bar{k}%
}m^{\bar{j}}m^rm^{\bar{k}}=R_{\bar{j}r\bar{k}h}m^{\bar{j}}m^rm^{\bar{k}}$
which leads to

$-(\frac 1F\bar{H}_{|0}+\frac 12\bar{H}(J+Y)+\bar{B}\bar{E}_{|0})l_h-(\bar{H}%
_{|s}m^s+\frac 12\bar{H}(V+H)+F\bar{B}\bar{E}_{|s}m^s)m_h$

$=-\bar{\mu}(m^rm_iL_{rh}^i)-FB\bar{\mu}(l^sL_{sh}^n)m_n.$

The transvection with $l^h$ and $m^h$ gives
\begin{eqnarray*}
\frac 1F\bar{H}_{|0}+\frac 12\bar{H}(J+Y)+\bar{B}\bar{E}_{|0} &=&Y_{|\bar{s}%
}m^{\bar{s}}+FBO_{|\bar{s}}m^{\bar{s}}-\frac 12FBO(\bar{V}+\bar{H}); \\
\bar{H}_{|s}m^s+\frac 12\bar{H}(V+H)+F\bar{B}\bar{E}_{|s}m^s &=&H_{|\bar{s}%
}m^{\bar{s}}+\frac 12H(\bar{V}+\bar{H})+FBE_{|\bar{s}}m^{\bar{s}}.
\end{eqnarray*}
Now, $R_{\bar{r}j\bar{h}0}m^{\bar{r}}m^j=-[F\bar{\lambda}%
(m^jl^km_iL_{jk}^i)+LB\bar{\lambda}(l^sl^kL_{sk}^n)m_n]l_{\bar{h}}$

$-[F\bar{\mu}(m^jl^km_iL_{jk}^i)+LB\bar{\mu}(l^sl^kL_{sk}^n)m_n]m_{\bar{h}}$

$=-[Y_{|\bar{0}}+FBO_{|\bar{0}}-\frac 12LO(\bar{J}+\bar{Y})]l_{\bar{h}}$

$-F[Y_{|\bar{s}}m^{\bar{s}}+FBO_{|\bar{s}}m^{\bar{s}}-\frac 12FBO(\bar{V}+%
\bar{H})]m_{\bar{h}}$

and

$R_{\bar{j}r\bar{0}h}m^{\bar{j}}m^r=-F\bar{\lambda}(m^rm_iL_{rh}^i)-LB\bar{%
\lambda}(l^sL_{sh}^n)m_n.$

The conjugation $\overline{R_{\bar{r}j\bar{h}0}m^{\bar{r}}m^jl^{\bar{h}}}%
=R_{r\bar{j}h\bar{0}}m^{\bar{j}}m^rl^h=R_{\bar{j}r\bar{0}h}m^{\bar{j}}m^rl^h$
gives
\begin{equation}
\bar{Y}_{|0}+F\bar{B}\bar{O}_{|0}-\frac 12L\bar{O}(J+Y)=Y_{|\bar{0}}+FBO_{|%
\bar{0}}-\frac 12LO(\bar{J}+\bar{Y}).  \label{2.14;;}
\end{equation}

\begin{lemma}
Let $(M,F)$ be a 2 - dimensional weakly K\"{a}hler complex Finsler space.
Then

i) $\frac 1F\bar{O}_{|0}-\frac 12\bar{O}(J+Y)-AO_{|\bar{0}}+\frac 12FAO(\bar{%
J}+\bar{Y})=\bar{J}_{|s}m^s;$

ii) $\frac 1F\bar{E}_{|0}-FAO_{|\bar{s}}m^{\bar{s}}+\frac 12FAO(\bar{V}+\bar{%
H})=V_{|\bar{s}}m^{\bar{s}}+\frac 12V(\bar{V}+\bar{H}).$
\end{lemma}

\begin{proof}
It results by Proposition 3.2, (\ref{2.13''}), (\ref{2.13;}) and (\ref{2.13;;;;}) . By computation using (\ref{c}), we obtain the global validity of these assertions.
\end{proof}

\begin{remark}
If $(M,F)$ is purely Hermitian ($A=0)$ and K\"{a}hler, then $\frac 1F\bar{O}%
_{|0}-\frac 12\bar{O}(J+Y)=\bar{J}_{|s}m^s$ and $V_{|\bar{s}}m^{\bar{s}%
}+\frac 12V(\bar{V}+\bar{H})=\frac 1F\bar{E}_{|0}.$
\end{remark}

In order to show the geometrical aspects of the above computations,
considering (\ref{1.8}), we define \textit{the horizontal holomorphic
sectional curvature} in direction $\lambda $ by
\begin{equation}
K_{F,\lambda }^h(z,\eta ):=2\mathbf{R(}\lambda ,\bar{\lambda},\lambda ,\bar{%
\lambda})  \label{2.15}
\end{equation}
and \textit{the horizontal holomorphic sectional curvature} in direction  $%
\mu $ by
\begin{equation}
K_{F,\mu }^h(z,\eta )=2\mathbf{R(}\mu ,\bar{\mu},\mu ,\bar{\mu}).
\label{2.15'}
\end{equation}

\begin{theorem}
Let $(M,F)$ be a 2 - dimensional complex Finsler space. Then

i) $K_{F,\lambda }^h(z,\eta )=2\mathbf{K},$ where $\mathbf{K:=-}\frac 1FJ_{|%
\bar{0}};$

ii) $K_{F,\mu }^h(z,\eta )=2\mathbf{W},$ where $\mathbf{W:=}-H_{|\bar{s}}m^{%
\bar{s}}-\frac 12H(\bar{V}+\bar{H})-BFE_{|\bar{s}}m^{\bar{s}}.$
\end{theorem}

\begin{proof}
By (\ref{2.12}) we obtain $\mathbf{R(}\lambda ,\bar{\lambda},\lambda ,\bar{%
\lambda})=l^{\bar{h}}l^{\bar{r}}l^jl^kR_{\bar{r}j\overline{h}k}=-\bar{\lambda%
}(l^jl^kl_iL_{jk}^i)=-\frac 1FJ_{|\bar{0}}$ and so \textit{i)} is proved. Similarly,
we have

$\mathbf{R(}\mu ,\bar{\mu},\mu ,\bar{\mu})=m^{\bar{h}}m^{\bar{r}}m^jm^kR_{%
\bar{r}j\overline{h}k}=-\bar{\mu}(L_{jk}^i)m^jm^km_i-FB\bar{\mu}%
(l^sL_{sk}^n)m^km_n$

$\mathbf{=}-\bar{\mu}(H)-\frac 12H(\bar{V}+\bar{H})-FB\bar{\mu}(E)=-H_{|\bar{%
s}}m^{\bar{s}}-\frac 12H(\bar{V}+\bar{H})-FBE_{|\bar{s}}m^{\bar{s}}\mathbf{,}
$ i.e. \textit{ii)}.

Changing the local coordinates $(z^k,\eta ^k)_{k=\overline{1,2}}$ into $%
(z^{\prime k},\eta ^{\prime k})_{k=\overline{1,2}}$, it results $\mathbf{K}%
^{\prime }=\mathbf{K}$ and $\mathbf{W}^{\prime }=\mathbf{W}$, which complete
the proof.
\end{proof}

We call the functions $\mathbf{K}$ and $\mathbf{W}$ the \textit{horizontal
curvature invariants}. Further on, our goal is to find the link between the $%
h\bar{h}-$ Riemann type tensors $R_{\bar{r}j\overline{h}k},$ $\mathbf{K}$
and $\mathbf{W.}$

Then, using (\ref{1.4'}), $\delta _i=l_i\lambda +m_i\mu $ and (\ref{2.12}), $%
R_{\bar{r}j\bar{h}k}=\mathbf{R}(\delta _j,\delta _{\bar{r}},\delta _k,\delta
_{\bar{h}})$ is decomposed into sixteen terms.

\begin{proposition}
Let $(M,F)$ be a 2 - dimensional complex Finsler space. Then
\begin{eqnarray}
R_{\bar{r}j\bar{h}k} &=&\mathbf{K}l_{\bar{r}}l_jl_{\bar{h}}l_k+\mathbf{W}m_{%
\bar{r}}m_jm_{\bar{h}}m_k  \label{2.16} \\
&&-[\frac 1F\bar{O}_{|0}-\frac 12\bar{O}(J+Y)]l_{\bar{r}}m_jl_{\bar{h}%
}l_k-[\frac 1FO_{|\bar{0}}-\frac 12O(\bar{J}+\bar{Y})]m_{\bar{r}}l_jl_{\bar{h%
}}l_k  \nonumber \\
&&-\bar{J}_{|s}m^sl_{\bar{r}}l_jl_{\bar{h}}m_k-J_{|\bar{s}}m^{\bar{s}}l_{%
\bar{r}}l_jm_{\bar{h}}l_k  \nonumber \\
&&-[V_{|\bar{s}}m^{\bar{s}}+\frac 12V(\bar{V}+\bar{H})]l_{\bar{r}}l_jm_{\bar{%
h}}m_k-\frac 1F\bar{E}_{|0}m_{\bar{r}}l_jl_{\bar{h}}m_k  \nonumber \\
&&-\frac 1FE_{|\bar{0}}l_{\bar{r}}m_jm_{\bar{h}}l_k-[\frac 1FY_{|\bar{0}%
}+BO_{|\bar{0}}-\frac 12FBO(\bar{J}+\bar{Y})]m_{\bar{r}}m_jl_{\bar{h}}l_k
\nonumber \\
&&-E_{|\bar{s}}m^{\bar{s}}m_{\bar{r}}l_jm_{\bar{h}}m_k-[\frac 1F\bar{H}%
_{|0}+\frac 12\bar{H}(J+Y)+\bar{B}\bar{E}_{|0}]m_{\bar{r}}m_jm_{\bar{h}}l_k
\nonumber \\
&&-\bar{E}_{|s}m^sl_{\bar{r}}m_jm_{\bar{h}}m_k-[\frac 1FH_{|\bar{0}}+\frac
12H(\bar{J}+\bar{Y})+BE_{|\bar{0}}]m_{\bar{r}}m_jl_{\bar{h}}m_k  \nonumber \\
&&-[O_{|\bar{s}}m^{\bar{s}}-\frac 12O(\bar{V}+\bar{H})]m_{\bar{r}}l_jm_{\bar{%
h}}l_k-[\bar{O}_{|s}m^s-\frac 12\bar{O}(V+H)]l_{\bar{r}}m_jl_{\bar{h}}m_k.
\nonumber
\end{eqnarray}
\end{proposition}

\begin{remark}
If $R_{\bar{r}j\bar{h}k}=0$ then the horizontal holomorphic sectional
curvature in any direction is zero.
\end{remark}

Now we come back to the Antonelli-Shimada metric (\ref{IV.9}) in order to
study its horizontal curvature invariants $\mathbf{K}$ and $\mathbf{W}$ and
its $h\overline{h}-$ Riemann type tensor. After some direct computations we
get
\begin{eqnarray*}
J &=&Y=2\frac{\partial \sigma }{\partial z^i}\eta ^i;\;E=O=0; \\
H &=&V=-\frac{\sqrt{2}}{|\eta ||\theta |F_{AS}}\left( \frac{\partial \sigma
}{\partial z}|\theta |^2\overline{\theta }-\frac{\partial \sigma }{\partial w%
}|\eta |^2\overline{\eta }\right) ,
\end{eqnarray*}
which substituted into (\ref{2.16}) give
\[
R_{\overline{r}j\overline{h}k}=g_{j\overline{r}}(\mathbf{K}l_{\overline{h}%
}l_k+\mathbf{W}m_{\overline{h}}m_k-\bar{J}_{|s}m^sl_{\bar{h}}m_k-J_{|\bar{s}%
}m^{\bar{s}}m_{\bar{h}}l_k),
\]
where
\begin{eqnarray}
\mathbf{K} &=&\mathbf{-}\frac 2{L_{AS}}\frac{\partial ^2\sigma }{\partial
z^k\partial \overline{z}^h}\eta ^k\overline{\eta }^h;  \label{2.16_} \\
\mathbf{W} &=&\mathbf{-}\frac{\mathbf{K}}2-\frac{e^{-4\sigma }L_{AS}}{|\eta
|^2|\theta |^2}\left( \frac{\partial ^2\sigma }{\partial z\partial \overline{%
z}}|\theta |^2+\frac{\partial ^2\sigma }{\partial w\partial \overline{w}}%
|\eta |^2\right) .  \nonumber
\end{eqnarray}

\begin{proposition}
The horizontal holomorphic sectional curvature in direction $\lambda $ of
the Antonelli-Shimada metric, $K_{F_{AS},\lambda }^h(z,\eta )$ is strictly
negative (positive) if and only if the $(1,1)-$ form $\frac{\partial
^2\sigma }{\partial z^k\partial \overline{z}^h}\eta ^k\overline{\eta }^h$ is
positive (negative) definite.
\end{proposition}

\begin{proof}
Indeed, $K_{F_{AS},\lambda }^h(z,\eta )=2\mathbf{K}=\mathbf{-}\frac 4{L_{AS}}%
\frac{\partial ^2\sigma }{\partial z^k\partial \overline{z}^h}\eta ^k%
\overline{\eta }^h.$ Its sign depends on the sign of the $(1,1)-$ form $%
\frac{\partial ^2\sigma }{\partial z^k\partial \overline{z}^h}\eta ^k%
\overline{\eta }^h.$
\end{proof}

For example if $\sigma (z,w)=\log \frac 1{\left( 1-|z|^2\right) \left(
|z|^2-|w|^2\right) }$ then $\frac{\partial ^2\sigma }{\partial z^k\partial
\overline{z}^h}\eta ^k\overline{\eta }^h$ is a purely Hermitian metric on
the Hartogs triangle $D=\left\{ (z,w)\in \mathbf{C}^2,\;|w|<|z|<1\right\} .$
Therefore, $K_{F_{AS},\lambda }^h(z,\eta )<0.$

Another example, if $\sigma (z,w)=\log (1-|z|^2-|w|^2),$ (it) leads to the
Bergman metric $-\frac{\partial ^2\sigma }{\partial z^k\partial \overline{z}%
^h}\eta ^k\overline{\eta }^h$ on the unit disk $D^2:=\left\{ (z,w)\in
\mathbf{C}^2,\;|z|^2+|w|^2<1\right\} .$ It results that $K_{F_{AS},\lambda
}^h(z,\eta )>0$ and $K_{F_{AS},\mu }^h(z,\eta )<0.$

\begin{proposition}
If $\sigma (z,w)$ is a harmonic function, i.e. $\frac{\partial ^2\sigma }{%
\partial z\partial \overline{z}}=\frac{\partial ^2\sigma }{\partial
w\partial \overline{w}}=0,$ then $K_{F_{AS},\mu }^h(z,\eta )=-\frac
12K_{F_{AS},\lambda }^h(z,\eta ).$
\end{proposition}

\begin{proof}
It results by (\ref{2.16_}).
\end{proof}

We point out that the $h\bar{h}-$ Riemann type tensors $R_{\bar{r}j\overline{%
h}k}$ generally are not symmetric. But, (\ref{2.16}) permits us to study
this particular case and some others.

\subsubsection{A weakly symmetry condition}

We call the property $R_{\bar{0}k\bar{0}0}=R_{\bar{0}0\bar{0}k}$ as being
\textit{a weakly symmetry }condition of the curvature. First, we want to see
what does this condition mean, in terms of the horizontal terms (\ref{2.11}%
). The answer is below.

\begin{corollary}
Let $(M,F)$ be a 2 - dimensional complex Finsler space. Then $R_{\bar{0}k%
\bar{0}0}=R_{\bar{0}0\bar{0}k}$ if and only if $\bar{J}_{|s}m^s=\frac 1F%
\bar{O}_{|0}-\frac 12\bar{O}(J+Y).$
\end{corollary}

\begin{proof}
(\ref{2.16}) gives that $R_{\bar{0}k\bar{0}0}=F^3\{\mathbf{K}l_k-[\frac 1F\bar{O}_{|0}-\frac
12\bar{O}(J+Y)]m_k\}$ and $R_{\bar{0}0\bar{0}k}=F^3[\mathbf{K}l_k-\bar{J}_{|s}m^sm_k].$
So (that), $R_{\bar{0}k\bar{0}0}=R_{\bar{0}0\bar{0}k}$ iff $\bar{J}_{|s}m^s=\frac 1F\bar{O}_{|0}-\frac
12\bar{O}(J+Y)$, which is globally by (\ref{c}).
\end{proof}

\begin{proposition}
Let $(M,F)$ be a 2 - dimensional weakly K\"{a}hler complex Finsler space
with $R_{\bar{0}k\bar{0}0}=R_{\bar{0}0\bar{0}k}$. Then it is either purely
Hermitian or with $J_{|\bar{s}}m^{\bar{s}}=0.$
\end{proposition}

\begin{proof}
Lemma 4.2 i) together with Corollary 4.1 leads to $AJ_{|\bar{s}}m^{\bar{s}}=0$, which gives $A=0$ or $J_{|\bar{s}}m^{\bar{s}}=0. $ Both conditions work globally.
\end{proof}

Further on, our goal is to find the properties of the invariants $\mathbf{K}$
and $\mathbf{W}$ with this condition of weakly symmetry. Therefore, we write
the identity i) of Proposition 2.2 for $J$%
\begin{equation}
J|_{\bar{r}|\bar{s}}-J_{|\bar{s}}|_{\bar{r}}=C_{\bar{s}\bar{r}}^{\bar{n}}J_{|%
\bar{n}}.  \label{2.16'''}
\end{equation}

Using 4.2 i) and (\ref{2.5''}) we have
\begin{equation}
J|_{\bar{r}|\bar{s}}=-\frac 1{2F}J_{|\bar{s}}l_{\bar{r}}+[\frac 1FO_{|\bar{s}%
}-\frac 1{2F}O(\bar{J}+\bar{Y})l_{\bar{s}}-\frac 1{2F}O(\bar{V}+\bar{H})m_{%
\bar{s}}]m_{\bar{r}}.  \label{2.16'}
\end{equation}

On the other hand, $J_{|\bar{s}}=-\mathbf{K}l_{\bar{s}}+J_{|\bar{h}}m^{\bar{h%
}}m_{\bar{s}}$ and using (\ref{2.4''}) we obtain
\begin{equation}
J_{|\bar{s}}|_{\bar{r}}=-\left( \mathbf{K}|_{\bar{r}}+\frac 1{2F}\mathbf{K}%
l_{\bar{r}}\mathbf{+}\frac 1FJ_{|\bar{r}}\right) l_{\bar{s}}+J_{|\bar{h}}|_{%
\bar{r}}m^{\bar{h}}m_{\bar{s}}.  \label{2.16''}
\end{equation}

Plugging (\ref{2.16'}) and (\ref{2.16''}) into (\ref{2.16'''}), it results

$-\frac 1{2F}J_{|\bar{s}}l_{\bar{r}}+\frac 1FO_{|\bar{s}}m_{\bar{r}}+\left(
\mathbf{K}|_{\bar{r}}+\frac 1{2F}\mathbf{K}l_{\bar{r}}\mathbf{+}\frac 1FJ_{|%
\bar{r}}-\frac 1{2F}O(\bar{J}+\bar{Y})m_{\bar{r}}\right) l_{\bar{s}}$

$-\left( J_{|\bar{h}}|_{\bar{r}}m^{\bar{h}}+\frac 1{2F}O(\bar{V}+\bar{H})m_{%
\bar{r}}\right) m_{\bar{s}}=(-\bar{A}\mathbf{K}+\bar{B}J_{|\bar{h}}m^{\bar{h}%
})m_{\bar{r}}m_{\bar{s}}$

which contracted by $l^{\bar{s}}$ and $m^{\bar{s}}m^{\bar{r}}$ respectively,
leads to
\begin{eqnarray}
\mathbf{K}|_{\bar{r}} &=&-\frac 1F[J_{|\bar{h}}m^{\bar{h}}+\frac 1FO_{|\bar{0%
}}-\frac 12O(\bar{J}+\bar{Y})]m_{\bar{r}};  \label{2.16..} \\
\bar{A}\mathbf{K} &=&J_{|\bar{h}}|_{\bar{r}}m^{\bar{h}}m^{\bar{r}}+\bar{B}%
J_{|\bar{h}}m^{\bar{h}}-\frac 1FO_{|\bar{s}}m^{\bar{s}}+\frac 1{2F}O(\bar{V}+%
\bar{H}).  \nonumber
\end{eqnarray}

Transvecting the Bianchi identity
\begin{equation}
\mathcal{A}_{kl}\left\{ R_{\bar{r}j\bar{h}k|l}-P_{\bar{r}j\bar{s}k}R_{\bar{0}%
l\overline{h}}^{\overline{s}}\right\} +R_{\bar{r}j\bar{h}n}T_{kl}^n=0,
\label{3.2'}
\end{equation}
(see \cite{Mub}, p. 77), by $\bar{\eta}^r\eta ^j\bar{\eta}^h\eta ^k$ it
follows
\begin{equation}
F\mathbf{K}_{|l}-\mathbf{K}_{|0}l_l-\bar{J}_{|s|0}m^sm_l+F\mathbf{K}%
l^kl_nT_{kl}^n-F\bar{J}_{|s}m^sl^km_nT_{kl}^n=0.  \label{2.16...}
\end{equation}

\begin{theorem}
Let $(M,F)$ be a connected 2 - dimensional weakly K\"{a}hler complex Finsler
space with $R_{\bar{0}k\bar{0}0}=R_{\bar{0}0\bar{0}k}$ and $|A|\neq 0$. Then
$\mathbf{K}$ is a constant on $(M,F).$
\end{theorem}

\begin{proof}
By Proposition 4.10 and by the first relation
of (\ref{2.16..}) it results that $\mathbf{K}|_{\bar{r}%
}=0$, i.e. $\mathbf{K}$ does not depend on $\eta $. Because $(M,F)$ is weakly K\"{a}hler, the identity (\ref{2.16...}) together with $\bar{J}_{|s|0}m^s=0$ gives $\mathbf{K}_{|l}=\frac
1F\mathbf{K}_{|0}l_l.$ But, using ii) from Proposition 2.2, we have $0=%
\mathbf{K}|_{\bar{k}|j}=\mathbf{K}_{|j}|_{\bar{k}}.$ On the other hand,

$\mathbf{K}_{|l}|_{\bar{r}}=-\frac 1{2L}\mathbf{K}_{|0}l_ll_{\bar{r}}+\frac
1F\mathbf{K}_{|0}|_{\bar{r}}l_l+\frac 1{F}\mathbf{K}_{|0}(\frac 1{2F}l_ll_{\bar{r}}+\frac 1{F}m_lm_{\bar{r}})$.

It follows that $\mathbf{K}_{|0}=0$  and so $\mathbf{K}_{|l}=0$, which is equivalent to $\frac{\partial
\mathbf{K}}{\partial z^l}=0,$ i.e. $\mathbf{K}$ is a constant on $(M,F).$
\end{proof}

\begin{corollary}
Let $(M,F)$ be a connected 2 - dimensional complex Finsler space with $R_{%
\bar{0}k\bar{0}0}=R_{\bar{0}0\bar{0}k}$. Then $\mathbf{K}$ depends on $z$
only if and only if $\bar{J}_{|s}m^s=0.$ Moreover, given any of these
equivalent conditions, we have $F\bar{A}\mathbf{K}=-O_{|\bar{s}}m^{\bar{s}%
}+\frac 12O(\bar{V}+\bar{H}).$
\end{corollary}

\begin{proof}
By Corollary 4.1, the first relation of (\ref{2.16..}) is $\mathbf{K}|_{\bar{r}}=-\frac 2FJ_{|\bar{h}}m^{%
\bar{h}}m_{\bar{r}}$ which justifies the above equivalence. $\bar{J}_{|s}m^s=0$ with the second equation of (%
\ref{2.16..}) give $F\bar{A}\mathbf{K}=-O_{|\bar{s}}m^{\bar{s}%
}+\frac 12O(\bar{V}+\bar{H}).$ Its global validity completes the proof.
\end{proof}

\begin{theorem}
Let $(M,F)$ be a connected 2 - dimensional complex Finsler space with $R_{%
\bar{0}k\bar{0}0}=R_{\bar{0}0\bar{0}k},$ $|A|\neq 0$ and $\mathbf{K}$ a
nonzero constant on $M.$ Then $(M,F)$ is weakly K\"{a}hler.
\end{theorem}

\begin{proof}
Substituting $\mathbf{K}_{|l}=\bar{J}_{|s}m^s=0$ in the relation (\ref{2.16...}) it follows that $L\mathbf{K}l^kl_nT_{kl}^n=0.$
Consequently, $l^kl_nT_{kl}^n=0,$ since $\mathbf{K}\neq 0\mathbf{.}$
\end{proof}

The purely Hermitian case is characterized by

\begin{proposition}
If $(M,F)$ is a 2 - dimensional K\"{a}hler purely Hermitian space, then $R_{%
\bar{0}k\bar{0}0}=R_{\bar{0}0\bar{0}k}.$
\end{proposition}

\begin{proof}
It results by Remark 4.2.
\end{proof}

\begin{theorem}
Let $(M,F)$ be a connected 2 - dimensional K\"{a}hler purely Hermitian
space. Then $\mathbf{K}$ is a constant on $(M,F)$ if and only if $J_{|\bar{h}%
}m^{\bar{h}}=0.$
\end{theorem}

\begin{proof}
By (\ref{2.16..}) we have $\mathbf{K}|_{\bar{r}}=-\frac 2FJ_{|\bar{h}}m^{%
\bar{h}}m_{\bar{r}}.$ If $\mathbf{K}$ is a constant on $(M,F)$
then $J_{|\bar{h}}m^{\bar{h}}=0.$ Conversely, if $J_{|\bar{h}}m^{\bar{h}}=0$
then $\mathbf{K}|_{\bar{r}}=0$ and by same arguments as in Theorem 4.8 it
results that $\mathbf{K}$ is a constant on $(M,F).$
\end{proof}

We note that the above Theorems give the necessary and sufficient conditions
that a connected 2 - dimensional complex Finsler space should be of constant
horizontal holomorphic curvature in direction $\lambda $. It is interesting
for us to see what happens with the horizontal holomorphic curvature in
direction $\mu $, in this case.

\begin{proposition}
Let $(M,F)$ be a connected 2 - dimensional weakly K\"{a}hler complex Finsler
space with $R_{\bar{0}k\bar{0}0}=R_{\bar{0}0\bar{0}k}$ and $|A|\neq 0$. Then

\textbf{\ }i) $F^2A\mathbf{K}=\Phi _{|0}-F(J+Y)\Phi ;$

ii) $E_{|\bar{0}}=-\frac{F\mathbf{K}}2\left( 1+A\bar{A}L\right) ;$

iii) $V_{|\bar{s}}m^{\bar{s}}+\frac 12V(\bar{V}+\bar{H})=-\frac{\mathbf{K}}%
2\left( 1-A\bar{A}L\right) ;$

iv) $Y_{|\bar{0}}=-\frac{F\mathbf{K}}2\left( 1-A\bar{A}L\right) +F\left|
\Phi \right| ^2;$

v) $H_{|\bar{0}}+\frac F2H(\bar{J}+\bar{Y})=-\frac{ALF\mathbf{K}}2\left(
\bar{B}-F\bar{A}B\right) +\frac{\bar{A}L}2[\Phi _{|k}m^k-(V+H)\Phi ]+L\Phi
\bar{\Omega};$

vi) $E_{|\bar{s}}m^{\bar{s}}=-\frac{F\mathbf{K}}2\left( \bar{B}-F\bar{A}%
B\right) -\frac{AF}2[\bar{\Phi}_{|\bar{k}}m^{\bar{k}}-(\bar{V}+\bar{H})\bar{%
\Phi}];$

vii) $\mathbf{W}=\mathbf{K}\left( 1+A\bar{A}L\right) -F(E_{|\bar{s}}m^{\bar{s%
}})|_lm^l-\frac 32BFE_{|\bar{s}}m^{\bar{s}}-L\left| \Omega \right| ^2,$

where $\Phi :=A_{|\bar{0}}+AF(\bar{J}+\bar{Y})$ and $\Omega :=A_{|\bar{k}}m^{%
\bar{k}}+A(\bar{V}+\bar{H}).$
\end{proposition}

\begin{proof}
Let us consider the Bianchi identity, (see \cite{Mub}, p.
77),
\begin{equation}
R_{\bar{r}j\overline{h}k}|_l-\Xi _{\bar{r}j\overline{h}l|k}-P_{\bar{r}j\bar{s%
}k}P_{\bar{0}l\overline{h}}^{\overline{s}}+S_{\bar{r}j\bar{s}l}R_{\bar{0}k%
\overline{h}}^{\bar{s}}+R_{\bar{r}j\overline{h}n}C_{kl}^n=0.  \label{3.2}
\end{equation}

In order to prove the statements i)-vii), we use Theorem 4.8, the covariant
derivatives (\ref{2.4''}), (\ref{2.5''}) and the expressions of the $v\bar{v}%
-,$ $h\bar{v}-,$ $v\bar{h}-,$ $h\bar{h}-$ Riemann type tensors.

Contracting into (\ref{3.2}) by $\bar{\eta}^rm^j\bar{\eta}^h\eta ^k$, using

$R_{\bar{r}j\overline{h}k}|_l\bar{\eta}^rm^j\bar{\eta}^h\eta ^k=-R_{\bar{0}j%
\overline{0}l}m^j=F^2[\bar{O}_{|s}m^s-\frac 12\bar{O}(V+H)]m_l=-F^3A\mathbf{K%
}m_l;$

$P_{\bar{r}j\bar{s}k}\bar{\eta}^r=S_{\bar{r}j\bar{s}l}\bar{\eta}%
^r=C_{kl}^n\eta ^k=0$ and

$\Xi _{\bar{r}j\overline{h}l|k}\bar{\eta}^rm^j\bar{\eta}^h\eta ^k=-F[\Phi
_{|0}-F(J+Y)\Phi ]m_l,$ we obtain i).

The contraction with $\bar{\eta}^r\eta ^j\bar{m}^h\eta ^k$ of (\ref{3.2}),

$R_{\bar{r}j\overline{h}k}|_l\bar{\eta}^r\eta ^j\bar{m}^h\eta ^k=-R_{\bar{0}l%
\bar{h}0}\bar{m}^h-R_{\bar{0}0\bar{h}l}\bar{m}^h+\frac 1LR_{\bar{0}0\bar{0}%
0}m_l$

$=F^2[\frac 1F\bar{E}_{|0}+V_{|\bar{s}}m^{\bar{s}}+\frac 12V(\bar{V}+\bar{H}%
)+\mathbf{K}]m_l$ and $\Xi _{\bar{r}j\overline{h}k}\eta ^j=0$ lead to
\[
\frac 1F\bar{E}_{|0}+V_{|\bar{s}}m^{\bar{s}}+\frac 12V(\bar{V}+\bar{H})=-%
\mathbf{K}.
\]
On the other hand, by Lemma 4.2 ii),
\[
\frac 1F\bar{E}_{|0}-V_{|\bar{s}}m^{\bar{s}}-\frac 12V(\bar{V}+\bar{H})=-LA%
\bar{A}\mathbf{K}.
\]
The last two relations give ii) and iii).

Now, contracting again (\ref{3.2}) by $\bar{m}^r\eta ^j\bar{\eta}^h\eta ^k,$
we have

$R_{\bar{r}j\overline{h}k}|_l\bar{m}^r\eta ^j\bar{\eta}^h\eta ^k=F^2(\mathbf{%
K}+\frac 1FY_{|\bar{0}}+\frac 1FE_{|\bar{0}})m_l$ and

$P_{\bar{r}j\bar{s}k}P_{\bar{0}l\overline{h}}^{\overline{s}}\bar{m}^r\eta ^j%
\bar{\eta}^h\eta ^k=F^2\left| \Phi \right| ^2m_l.$

It results $[\mathbf{K}+\frac 1FY_{|\bar{0}}+\frac 1FE_{|\bar{0}}-\left|
\Phi \right| ^2]m_l=0$. Hereby, $Y_{|\bar{0}}=-\mathbf{K}F-E_{|\bar{0}%
}+F\left| \Phi \right| ^2$, which together with ii) implies iv).

Next we prove v) and vi). First we contract (\ref{3.2}) with $\bar{\eta}^rm^j%
\bar{\eta}^hm^km^l$ and we obtain

$(R_{\bar{0}j\bar{0}k}m^jm^k)|_lm^l-BR_{\bar{0}j\bar{0}k}m^jm^k-\Xi _{\bar{0}%
j\overline{0}l|k}m^jm^km^l+R_{\bar{0}j\bar{0}n}C_{kl}^nm^jm^km^l=0.$ This
implies that
\begin{equation}
A|_lm^lL\mathbf{K}=-[\Phi _{|k}m^k+(V+H)\Phi ]  \label{9}
\end{equation}

The contraction of (\ref{3.2}) by $\bar{m}^r\eta ^j\bar{\eta}^hm^km^l$
implies

$(R_{\bar{r}0\bar{0}k}\bar{m}^rm^k)|_lm^l-(R_{\bar{r}l\bar{0}k}\bar{m}%
^rm^k+P_{\bar{r}0\bar{s}k}P_{\bar{0}l\overline{0}}^{\overline{s}}\bar{m}%
^rm^k-R_{\bar{r}0\overline{0}n}C_{kl}^n\bar{m}^rm^k)m^l=0,$ which gives

$H_{|\bar{0}}+\frac F2H(\bar{J}+\bar{Y})=F\bar{E}_{|0}|_lm^l+L\Phi \bar{%
\Omega}.$ Now, this together with ii), (\ref{9}) and (\ref{2.5'}) gives v).

The contraction of (\ref{3.2}) by $\bar{m}^r\eta ^j\bar{m}^h\eta ^km^l$ gives

$(R_{\bar{r}0\bar{h}0}\bar{m}^r\bar{m}^h)|_lm^l+BR_{\bar{r}0\bar{h}0}\bar{m}%
^r\bar{m}^h-R_{\bar{r}l\bar{h}0}\bar{m}^r\bar{m}^hm^l-R_{\bar{r}0\bar{h}l}%
\bar{m}^r\bar{m}^hm^l$

$-P_{\bar{r}0\bar{s}0}P_{\bar{0}l\overline{h}}^{\overline{s}}\bar{m}^r\bar{m}%
^hm^l=0,$ which is equivalent to

$L\mathbf{K}\bar{A}|_lm^l+\frac 1F\bar{H}_{|0}+\frac 12\bar{H}(J+Y)+\bar{B}%
\bar{E}_{|0}+E_{|\bar{s}}m^{\bar{s}}+B\bar{A}L\mathbf{K}=F\bar{\Phi}\Omega .$

Using ii), v) and (\ref{2.5'}) it leads to vi).

For vii) we contract (\ref{3.2}) with $\bar{m}^r\eta ^j\bar{m}^hm^km^l$ and
we deduce

$(R_{\bar{r}0\bar{h}k}\bar{m}^r\bar{m}^hm^k)|_lm^l+\frac B2R_{\bar{r}0\bar{h}%
k}\bar{m}^r\bar{m}^hm^k$

$+\frac 1LR_{\bar{0}0\bar{h}k}\bar{m}^hm^k-R_{\bar{r}l\bar{h}k}\bar{m}^r\bar{%
m}^hm^km^l+\frac 1LR_{\bar{r}0\bar{0}k}\bar{m}^rm^k$

$-P_{\bar{r}0\bar{s}k}P_{\bar{0}l\overline{h}}^{\overline{s}}\bar{m}^r\bar{m}%
^hm^km^l+R_{\bar{r}0\overline{h}n}C_{kl}^n\bar{m}^r\bar{m}^hm^km^l=0.$

From here we obtain

$-F(E_{|\bar{s}}m^{\bar{s}})|_lm^l-\frac B2E_{|\bar{s}}m^{\bar{s}}-V_{|\bar{s%
}}m^{\bar{s}}-\frac 12V(\bar{V}+\bar{H})-\mathbf{W}-\frac 1F\bar{E}_{|0}+A%
\bar{A}L\mathbf{K}-BFE_{|\bar{s}}m^{\bar{s}}=L\left| \Omega \right| ^2,$
which leads to vii).
The global validity of the above statements results by straightforward computations using (\ref{c}).
\end{proof}

Next, we establish some consequences of the above Proposition. From vii) and
Theorem 4.7 ii) it immediately results

\begin{corollary}
Let $(M,F)$ be a connected 2 - dimensional weakly K\"{a}hler complex Finsler
space with $R_{\bar{0}k\bar{0}0}=R_{\bar{0}0\bar{0}k}$ and $|A|\neq 0$. Then
\[
K_{F,\mu }^h(z,\eta )=2\mathbf{K}\left( 1+A\bar{A}L\right) -2F(E_{|\bar{s}%
}m^{\bar{s}})|_lm^l-3BFE_{|\bar{s}}m^{\bar{s}}-2L\left| \Omega \right| ^2.
\]
\end{corollary}

So, we remark that the conditions which assure that $K_{F,\lambda }^h(z,\eta
)$ is a constant on $M$ do not suffice to imply that $K_{F,\mu }^h(z,\eta )$
is a constant, too.

Proposition 4.12 i) gives

\begin{corollary}
Let $(M,F)$ be a connected 2 - dimensional weakly K\"{a}hler complex Finsler
space with $R_{\bar{0}k\bar{0}0}=R_{\bar{0}0\bar{0}k}$ and $|A|\neq 0$. Then
$\mathbf{K}=0$ if and only if $\Phi _{|0}=F(J+Y)\Phi .$
\end{corollary}

\begin{corollary}
Let $(M,F)$ be a connected 2 - dimensional weakly K\"{a}hler complex Finsler
space with $R_{\bar{0}k\bar{0}0}=R_{\bar{0}0\bar{0}k}$ and $|A|\neq 0$. If $%
\Phi _{|k}m^k=(V+H)\Phi $ then
\begin{equation}
\mathbf{W}=\mathbf{K}\left( 1+A\bar{A}L\right) -\frac{L\mathbf{K}}2(\mathbf{I%
}+\frac 12F\bar{A}B^2+F\bar{A}B|_km^k)-L\left| \Omega \right| ^2.  \label{22}
\end{equation}
\end{corollary}

\begin{proof}
By Proposition 4.12 vi) we have $E_{|\bar{s}}m^{\bar{s}}=-%
\frac{F\mathbf{K}}2\left( \bar{B}-F\bar{A}B\right) $, which substituted into
vii) gives (\ref{22}), taking into account (\ref{2.5'}). Their global validity complete the proof.
\end{proof}

\begin{remark}
Under assumptions of above Corollary and $B=0,$ contracting (\ref{3.2}) with
$\bar{m}^rm^j\bar{m}^hm^km^l$ we obtain $\mathbf{K}=0.$ On the other hand,
by Proposition 4.1, $\mathbf{I}=0.$ Therefore, (\ref{22}) becomes $\mathbf{W}%
=-L\left| \Omega \right| ^2$. Further more $\mathbf{W}=0$ if and only if $%
\Omega =0.$
\end{remark}

\begin{theorem}
Let $(M,F)$ be a connected 2 - dimensional weakly K\"{a}hler complex Finsler
space with $R_{\bar{0}k\bar{0}0}=R_{\bar{0}0\bar{0}k}$, $|A|\neq 0$ and
\[
\Phi _{|k}=\Phi [(J+Y)l_k+(V+H)m_k].
\]
Then
\[
K_{F,\lambda }^h(z,\eta )=0\mbox{ and }K_{F,\mu }^h(z,\eta )=-2L\left|
\Omega \right| ^2\leq 0.
\]

Moreover, $K_{F,\mu }^h(z,\eta )$ is a constant if and only if $\Omega
=\frac cF,$ where $c\in \mathbf{C}.$
\end{theorem}

\begin{proof}
It results by Corollary 4.4 and Proposition 4.12 vi) and
vii).
\end{proof}

\begin{theorem}
If $(M,F)$ is a connected 2 - dimensional complex Berwald space with $R_{%
\bar{0}k\bar{0}0}=R_{\bar{0}0\bar{0}k}$ and $|A|\neq 0$, then $R_{\bar{r}j%
\bar{h}k}=0.$
\end{theorem}

\begin{proof}
The assumption of Berwald leads to $\Phi =\Omega =0.$
Applying now Proposition 4.12 and (\ref{2.16}) we obtain $R_{\bar{r}j\bar{h}%
k}=0.$
\end{proof}

\begin{remark}
If $(M,F)$ is a connected 2 - dimensional complex Berwald space with $R_{%
\bar{0}k\bar{0}0}=R_{\bar{0}0\bar{0}k}$ and $|A|\neq 0$, then the horizontal
holomorphic sectional curvature in any direction is zero.
\end{remark}

Moreover, we obtain a known result from Hermitian geometry

\begin{theorem}
If $(M,F)$ is a 2 - dimensional K\"{a}hler purely Hermitian space with $%
\mathbf{K}$ a constant on $M$, then $R_{\bar{r}j\bar{h}k}=\frac{\mathbf{K}}%
2(g_{j\bar{r}}g_{k\bar{h}}+g_{k\bar{r}}g_{j\bar{h}})$ and the horizontal
holomorphic sectional curvatures in directions $\lambda $ and $\mu $ are $2%
\mathbf{K}$.
\end{theorem}

\begin{proof}
Because $A=B=0,$ the Bianchi identity (\ref{3.2}) is $R_{%
\bar{r}j\overline{h}k}|_l=0.$ Some contractions of it lead to

$-O_{|\bar{s}}m^{\bar{s}}+\frac 12O(\bar{V}+\bar{H})=H_{|\bar{0}}+\frac F2H(%
\bar{J}+\bar{Y})=E_{|\bar{s}}m^{\bar{s}}=0$,

$\frac 1FE_{|\bar{0}}=V_{|\bar{s}}m^{\bar{s}}+\frac 12V(\bar{V}+\bar{H}%
)=\frac 1FY_{|\bar{0}}=-\frac{\mathbf{K}}2$ and $\mathbf{W}=\mathbf{K.}$

Therefore, (\ref{2.16}) become

$R_{\bar{r}j\bar{h}k}=\frac{\mathbf{K}}2\mathbf{(}2l_{\bar{r}}l_jl_{\bar{h}%
}l_k+2m_{\bar{r}}m_jm_{\bar{h}}m_k+l_{\bar{r}}l_jm_{\bar{h}}m_k+m_{\bar{r}%
}l_jl_{\bar{h}}m_k+l_{\bar{r}}m_jm_{\bar{h}}l_k$

$+m_{\bar{r}}m_jl_{\bar{h}}l_k)=\frac{\mathbf{K}}2(g_{j\bar{r}}g_{k\bar{h}%
}+g_{k\bar{r}}g_{j\bar{h}})$ and $K_{F,\lambda }^h(z,\eta )=K_{F,\mu
}^h(z,\eta )=2\mathbf{K}.$
\end{proof}

\subsubsection{A particular case}

Consider the problem of classifying the 2 - dimensional complex Finsler
spaces for which
\begin{equation}
R_{\bar{r}j\bar{h}k}=\mathcal{K}(g_{j\bar{r}}g_{k\bar{h}}+g_{k\bar{r}}g_{j%
\bar{h}}),  \label{3.1}
\end{equation}
where $\mathcal{K}=\mathcal{K}(z,\eta ):T^{\prime }M\rightarrow \mathbf{R}$.
We note that the 2 - dimensional complex Finsler spaces with (\ref{3.1}) has
the property that $K_{F,\lambda }^h(z,\eta )=K_{F,\mu }^h(z,\eta )=4\mathcal{%
K}$. To solve the stated problem we use the Bianchi identities (\ref{3.2})
and (\ref{3.2'}). Indeed, contracting the identity (\ref{3.2}) by $\bar{\eta}%
^r\eta ^j$ and taking into account

$R_{\bar{r}j\overline{h}k}|_l\bar{\eta}^r\eta ^j=R_{\bar{0}0\overline{h}%
k}|_l-R_{\bar{0}l\overline{h}k}$

$=\left[ L\mathcal{K}(2g_{k\bar{h}}-m_{\bar{h}}m_k)\right] |_l-F\mathcal{K}%
(2l_ll_{\bar{h}}l_k+l_lm_{\bar{h}}m_k+m_lm_{\bar{h}}l_k)$

$=F\mathcal{K}(2g_{k\bar{h}}-m_{\bar{h}}m_k)l_l+L\mathcal{K}|_l(2g_{k\bar{h}%
}-m_{\bar{h}}m_k)$

$+F\mathcal{K}l_km_{\bar{h}}m_l-F\mathcal{K}(2l_ll_{\bar{h}}l_k+l_lm_{\bar{h}%
}m_k+m_lm_{\bar{h}}l_k)$

$=L\mathcal{K}|_l(2g_{k\bar{h}}-m_{\bar{h}}m_k);$

$\Xi _{\bar{r}j\overline{h}l|k}\eta ^j=P_{\bar{r}j\bar{s}k}P_{\bar{0}l%
\overline{h}}^{\overline{s}}\bar{\eta}^r=S_{\bar{r}j\bar{s}l}R_{\bar{0}k%
\overline{h}}^{\bar{s}}\bar{\eta}^r=0$ and

$R_{\bar{r}j\overline{h}n}C_{kl}^n\bar{\eta}^r\eta ^j=L\mathcal{K}(2g_{n\bar{%
h}}-m_{\bar{h}}m_n)(Al^n+Bm^n)m_km_l$

$=L\mathcal{K}(2Al_{\bar{h}}+Bm_{\bar{h}})m_km_l,$ we obtain
\begin{equation}
\mathcal{K}|_l(2g_{k\bar{h}}-m_{\bar{h}}m_k)+\mathcal{K}(2Al_{\bar{h}}+Bm_{%
\bar{h}})m_km_l=0,  \label{3.3}
\end{equation}
which is true in any local coordinates. Moreover, contraction by $\bar{\eta}%
^h\eta ^k$ in (\ref{3.3}) it follows $\mathcal{K}|_l=0,$ which means that $%
\mathcal{K}=\mathcal{K}(z)$ is a function of $z$ alone. This implies that $%
\mathcal{K}(2Al_{\bar{h}}+Bm_{\bar{h}})=0$ and hence that either $A=0$
(i.e. the purely Hermitian case) for any $\mathcal{K}=\mathcal{K}(z)$, or $%
|A|\neq 0$ and $\mathcal{K}=0$ for any $|B|.$

So, we have proven

\begin{corollary}
Let $(M,F)$ be a connected 2 - dimensional complex Finsler space with $R_{%
\bar{r}j\bar{h}k}=\mathcal{K}(g_{j\bar{r}}g_{k\bar{h}}+g_{k\bar{r}}g_{j\bar{h%
}}).$ Then it is purely Hermitian with $\mathcal{K}$ a function of $z$ alone,
or it is not purely Hermitian but with $\mathcal{K}=0$.
\end{corollary}

For $\mathcal{K}\neq 0,$ (\ref{3.2'}) gives

$\mathcal{K}_{|l}(g_{j\bar{r}}g_{k\bar{h}}+g_{k\bar{r}}g_{j\bar{h}})-%
\mathcal{K}_{|k}(g_{j\bar{r}}g_{l\bar{h}}+g_{l\bar{r}}g_{j\bar{h}})+\mathcal{%
K}(g_{j\bar{r}}g_{n\bar{h}}+g_{n\bar{r}}g_{j\bar{h}})T_{kl}^n=0,$ which
contracted by $\bar{\eta}^r\eta ^j\bar{\eta}^h\eta ^k$ leads to
\begin{equation}
F\mathcal{K}_{|l}-\mathcal{K}_{|0}l_l+F\mathcal{K}l^kl_nT_{kl}^n=0.
\label{3.3'}
\end{equation}

\begin{theorem}
Let $(M,F)$ be a connected 2 - dimensional complex Finsler space with $R_{%
\bar{r}j\bar{h}k}=\mathcal{K}(g_{j\bar{r}}g_{k\bar{h}}+g_{k\bar{r}}g_{j\bar{h%
}}),$ $\mathcal{K}\neq 0.$ Then $\mathcal{K}$ is a constant on $M$ if and
only if $(M,F)$ is K\"{a}hler.
\end{theorem}

\begin{proof}
If $\mathcal{K}$ is a constant on $(M,F)$ then, by (\ref{3.3'}) it results $%
l^kl_nT_{kl}^n=0$ which leads to $T_{kl}^n=0,$ i.e. $(M,F)$ is K\"{a}hler.

Conversely, if $(M,F)$ is K\"{a}hler then by (\ref{3.3'}) we have $\mathcal{K%
}_{|l}=\frac 1F\mathcal{K}_{|0}l_l.$ Hence $\mathcal{K}_{|l}m^l=0.$ On the
other hand, using the identity ii) of Proposition 2.2 it results $0=%
\mathcal{K}|_{\bar{k}|j}=\mathcal{K}_{|j}|_{\bar{k}}.$ Therefore,

$0=\mathcal{K}_{|l}|_{\bar{r}}=-\frac 1{2L}\mathcal{K}_{|0}l_ll_{\bar{r}%
}+\frac 1F\mathcal{K}_{|0}|_{\bar{r}}l_l-\frac 1{2L}\mathcal{K}_{|0}l_ll_{%
\bar{r}}$ and so

$-\frac 1L\mathcal{K}_{|0}l_ll_{\bar{r}}=0$ which gives $\mathcal{K}_{|0}=0.$
It follows $\mathcal{K}_{|l}=0$ equivalently with $\frac{\partial \mathcal{K}%
}{\partial z^l}=0,$ i.e. $\mathcal{K}$ is a constant on $(M,F).$
\end{proof}

Now, we study the space with $|A|\neq 0$ and $\mathcal{K}=0$ for any $|B|.$
By (\ref{3.1}) results $R_{\bar{r}j\bar{h}k}=0,$ and so the horizontal
holomorphic sectional curvature in any direction vanishes identically.
Moreover, the identity (\ref{3.2}) become

\begin{equation}
\Xi _{\bar{r}j\overline{h}l|k}+P_{\bar{r}j\bar{s}k}P_{\bar{0}l\overline{h}}^{%
\overline{s}}=0,  \label{12}
\end{equation}
which contracted by $\bar{m}^r\eta ^j\bar{\eta}^h\eta ^km^l$ and $\bar{m}%
^r\eta ^j\bar{m}^hm^km^l$ leads to $\Phi =\Omega =0,$ which are globally.
Therefore, by Lemma 4.1 we obtain

\begin{theorem}
If $(M,F)$ is a connected 2 - dimensional complex Finsler space with $R_{%
\bar{r}j\bar{h}k}=0,$ $|A|\neq 0,$ then $\dot{\partial}_{\bar{h}}G^i=0.$
\end{theorem}

\begin{remark}
The converse of above Theorem is not true. There exists $2$ - dimensional
complex Finsler spaces with $\dot{\partial}_{\bar{h}}G^i=0$ and $R_{\bar{r}j%
\bar{h}k}\neq 0.$ A such example is given by the Antonelli - Shimada metric (%
\ref{IV.9}).
\end{remark}

Now, using Theorem 4.4 we have proven

\begin{corollary}
If $(M,F)$ is a connected 2 - dimensional weakly K\"{a}hler complex Finsler
space with $R_{\bar{r}j\bar{h}k}=0,$ $|A|\neq 0,$ then it is Berwald.
\end{corollary}

\textbf{Acknowledgment: }This paper is supported by the Sectorial
Operational Program Human Resources Development (SOP HRD), financed from the
European Social Fund and by Romanian Government under the Project number
POSDRU/89/1.5/S/59323.

\begin{flushleft}

Transilvania Univ.,
Faculty of Mathematics and Informatics

Iuliu Maniu 50, Bra\c{s}ov 500091, ROMANIA

e-mail: nicoleta.aldea@lycos.com

e-mail: gh.munteanu@unitbv.ro
\end{flushleft}

\end{document}